\input ./style/arxiv-vmsta.cfg
\documentclass[numbers,compress,v1.0.1]{vmsta}

\volume{4}
\issue{4}
\pubyear{2017}
\firstpage{353}
\lastpage{379}

\doi{10.15559/17-VMSTA90}

\usepackage{dcolumn}
\usepackage{mathbh}

\startlocaldefs
\newcommand{\lleft}{\left}
\newcommand{\rright}{\right}
\newcommand{\rrVert}{\Vert}
\newcommand{\llVert}{\Vert}
\newcommand{\rrvert}{\vert}
\newcommand{\llvert}{\vert}
\allowdisplaybreaks

\endlocaldefs

\newtheorem{thm}{Theorem}

\newtheorem{remark}{Remark}
\newtheorem{lemma}{Lemma}

\theoremstyle{definition}
\newtheorem{defin}{Definition}

\begin{document}

\begin{frontmatter}

\title{Double barrier reflected BSDEs with stochastic Lipschitz coefficient}

\author{\inits{M.}\fnm{Mohamed}\snm{Marzougue}\corref{cor1}}\email{md.marzougue@gmail.com}
\author{\inits{M.}\fnm{Mohamed}\snm{El Otmani}}\email{m.elotmani@uiz.ac.ma}
\cortext[cor1]{Corresponding author.}
\address{Laboratoire d'Analyse Math\'ematique et Applications (LAMA)\\
Facult\'{e} des Sciences Agadir, {Universit\'{e} Ibn Zohr}, {Maroc}}

\markboth{M. Marzougue, M. El Otmani}{Double barrier reflected BSDEs
with stochastic Lipschitz coefficient}

\begin{abstract}
This paper proves the existence and uniqueness of a solution to doubly
reflected backward stochastic differential equations where the coefficient
is stochastic Lipschitz, by means of the penalization method.
\end{abstract}

\received{21 July 2017}
\revised{14 November 2017}
\accepted{14 November 2017}
\publishedonline{8 December 2017}
\end{frontmatter}

\section{Introduction}
Backward Stochastic Differential Equations (BSDEs) were introduced (in
the non-linear case) by Pardoux and Peng \cite{PP}.
Precisely, given a data $(\xi,f)$ of a square integrable random
variable $\xi$ and a progressively measurable function $f$, a solution
to BSDE associated with data $(\xi,f)$ is a pair of $\mathcal
{F}_t$-adapted processes $(Y,Z)$ satisfying
\begin{equation}
\label{BSDE} Y_{t}=\xi+\int_{t}^{T}
f(s,Y_{s},Z_{s})ds-\int_{t}^{T}Z_{s}dB_{s},
\quad0\leq t\leq T.
\end{equation}
These equations have attracted great interest due to their connections
with mathematical finance \cite{EQ,EPeQ}, stochastic control and
stochastic games \cite{Bi,HL} and partial differential equations \cite
{Pa,PP2}.

In their seminal paper \cite{PP}, Pardoux and Peng generalized such
equations to the Lipschitz condition and proved existence and
uniqueness results in a Brownian framework. Moreover, many efforts have
been made to relax the Lipschitz condition on the coefficient. In this
context, Bender and Kohlmann \cite{BK} considered the so-called
stochastic Lipschitz condition introduced by El Karoui and Huang \cite{ELH}.

Further, El Karoui et al. \cite{EKPPQ} have introduced the notion of
reflected BSDEs (RBSDEs in short), which is a BSDE but the solution is
forced to stay above a lower barrier. In detail, a solution to such
equations is a triple of processes $(Y,Z,K)$ satisfying
\begin{equation}
\label{RBSDE} Y_{t}=\xi+\int_{t}^{T}
f(s,Y_{s},Z_{s})ds+K_T-K_t-\int
_{t}^{T}Z_{s}dB_{s},\quad
Y_t\geq L_t\ 0\leq t\leq T,
\end{equation}
where $L$, the so-called barrier, is a given stochastic process. The
role of the continuous increasing process $K$ is to push the state
process upward with the minimal energy, in order to keep it above $L$;
in this sense, it satisfies $\int_0^T(Y_t-L_t)dK_t=0$. 
The authors have proved that equation (\ref{RBSDE}) has a unique
solution under square integrability of the terminal condition $\xi$ and
the barrier $L$, and the Lipschitz property of the coefficient $f$.

RBSDEs have been proven to be powerful tools in mathematical finance
\cite{EPeQ}, mixed game problems \cite{CK}, providing a probabilistic
formula for the viscosity solution to an obstacle problem for a class
of parabolic partial differential equations \cite{EKPPQ}.

Later, Cvitanic and Karatzas \cite{CK} studied doubly reflected BSDEs
(DRBSDEs in short). A solution to such an equation related to a
generator $f$, a terminal condition $\xi$ and two barriers $L$ and $U$
is a quadruple of $(Y, Z,K^+, K^-)$ which satisfies
\begin{eqnarray}
\label{a} &&\hspace{-1cm}  %
\begin{cases}
Y_{t}=\xi+\displaystyle\int_{t}^{T}
f(s,Y_{s},Z_{s})ds+\bigl(K_{T}^{+}-K_{t}^{+}\bigr)-\bigl(K_{T}^{-}-K_{t}^{-}\bigr)-
\displaystyle\int_{t}^{T}Z_{s}dB_{s}  \\
L_{t}\leq Y_{t}\leq U_{t},\ \forall t\leq T \mbox{ and }\displaystyle\int_{0}^{T}(Y_{t}-L_{t})dK_{t}^{+}=\displaystyle\int_{0}^{T}(U_{t}-Y_{t})dK_{t}^{-}=0.
\end{cases}
\end{eqnarray}
In this case, a solution $Y$ has to remain between the lower barrier
$L$ and upper barrier $U$. This is achieved by the cumulative action of
two continuous, increasing reflecting processes $K^\pm$. The authors
proved the existence and uniqueness of the solution when $f(t,\omega
,y,z)$ is Lipschitz on $(y,z)$ uniformly in $(t,\omega)$. At the same
time, one of the barriers $L$ or $U$ is regular or they satisfy the
so-called Mokobodski condition, which turns out into the existence of a
difference of a non-negative supermartingales between $L$ and $U$. In
addition, many efforts have been made to relax the conditions on $f$,
$L$ and $U$ \cite{BHM,HH,HHd,LSM,LS,Xu,ZZ} or to deal with other issues
\cite{CM,EL,ELM,EOH,RE}.

Let us have a look at the pricing problem of an American game option
driven by Black--Scholes market model which is given by the following
system of stochastic differential equations
\begin{eqnarray*}
&&\lleft\{ %
\begin{array}{@{}ll}
dS^0_t=r(t)S^0_tdt, & \hbox{$S^0_0>0$;} \\
dS_t=S_t\bigl(\bigl(r(t)+\theta(t)\sigma(t)\bigr)dt+\sigma(t)dB_t\bigr), & \hbox{$S_0>0$,}
\end{array} %
\rright.
\end{eqnarray*}
where $r(t)$ is the interest rate process, $\theta(t)$ is the risk
premium process, $\sigma(t)$ is the volatility process of the market.
The fair price of the American game option is defined by\vadjust{\goodbreak}
\begin{equation*}
Y_t=\inf_{\tau\in\Im_{[0,T]}}\sup_{\nu\in\Im_{[0,T]}}
\mathbb{E} \bigl[e^{-r(t)\sigma(t)\wedge\theta(t)}J(\tau,\nu)|\mathcal{F}_t \bigr],
\end{equation*}
where $\Im_{[0,T]}$ is the collection of all stopping times $\tau$ with
values between $0$ and $T$, and $J$ is a \textit{Payoff} given by
\[
J(\tau,\nu)=U_{\nu}\mathbh{1}_{\{\nu<\tau\}}+L_{\tau}
\mathbh{1}_{\{\tau
\leq\nu\}}+\xi\mathbh{1}_{\{\nu\wedge\tau=T\}}.
\]
Here $r(t)$, $\sigma(t)$ and $\theta(t)$ are stochastic, moreover they
are not bounded in general. So the existence results of Cvitanic and
Karatzas \cite{CK}, Li and Shi \cite{LS} with completely separated
barriers cannot be applied.

Motivated by the above works, the purpose of the present paper is to
consider a class of DRBSDEs driven by a Brownian motion with stochastic
Lipschitz coefficient. We try to get the existence and uniqueness of
solutions to those DRBSDEs by means of the penalization method and the
fixed point theorem. Furthermore, the comparison theorem for the
solutions to DRBSDEs will be established.

The paper is organized as follows: in Section~\ref{s1}, we give some
notations and assumptions needed in this paper. In Section~\ref{s2}, we
establish the a priori estimates of solutions to DRBSDEs. In Section~\ref{s3},
we prove the existence and uniqueness of solutions to DRBSDEs
via penalization method when one barrier is regular, in the first
subsection, then we study the case when the barriers are completely
separated, in the second subsection. In Section~\ref{s4}, we give the
comparison theorem for the solutions to DRBSDEs. Finally, an Appendix
is devoted to the special case of RBSDEs with lower barrier when the
generator only depends on $y$; furthermore, the corresponding
comparison theorem will be established under the stochastic Lipschitz
coefficient.
\section{Notations}\label{s1}
Let $(\varOmega,\mathcal{F},(\mathcal{F}_{t})_{t\leq T},\mathbb{P})$ be a
filtered probability space. Let $(B_{t})_{t\leq T}$ be a
$d$-dimensional Brownian motion. We assume that $(\mathcal
{F}_{t})_{t\leq T}$ is the standard filtration generated by the
Brownian motion $(B_{t})_{t\leq T}$.

We will denote by $|.|$ the Euclidian norm on $\mathbb{R}^{d}$.

Let's introduce some spaces:
\begin{itemize}
\item$\mathcal{L}^{2}$ is the space of $\mathbb{R}$-valued and
$\mathcal{F}_{T}$-measurable random variables $\xi$ such that
\[
\|\xi\|^{2}=\mathbb{E} \bigl[|\xi|^{2} \bigr] < +\infty.
\]
\item$\mathcal{S}^{2}$ is the space of $\mathbb{R}$-valued and
$\mathcal{F}_t$-progressively measurable processes $(K_{t})_{t\leq T}$
such that
\[
\|K\|^{2}=\mathbb{E} \Bigl[\sup\limits
_{0\leq t\leq T}|K_{t}|^{2}
\Bigr] < +\infty.
\]
\end{itemize}
Let $\beta> 0$ and $(a_{t})_{t\leq T}$ be a non-negative $\mathcal
{F}_{t}$-adapted process. We define the increasing continuous process
$A(t)= \int_0^t a^2(s)ds$, for all $t\leq T$, and introduce the
following spaces:
\begin{itemize}
\item$\mathcal{L}^{2}(\beta,a)$ is the space of $\mathbb{R}$-valued
and $\mathcal{F}_{T}$-measurable random variables $\xi$ such that
\[
\|\xi\|_\beta^{2}=\mathbb{E} \bigl[e^{\beta A(T)}|
\xi|^{2} \bigr] < +\infty.
\]
\item$\mathcal{S}^{2}(\beta,a)$ is the space of $\mathbb{R}$-valued
and $\mathcal{F}_t$-adapted continuous processes $(Y_t)_{t\leq T}$ such that
\[
\|Y\|_\beta^{2}=\mathbb{E} \Bigl[\sup\limits
_{0\leq t\leq T}e^{\beta
A(t)}|Y_{t}|^{2}
\Bigr] < +\infty.
\]
\item$\mathcal{S}^{2,a}(\beta,a)$ is the space of $\mathbb{R}$-valued
and $\mathcal{F}_t$-adapted processes $(Y_t)_{t\leq T}$ such that
\[
\|aY\|_\beta^{2}=\mathbb{E} \Biggl[\int
_{0}^{T}e^{\beta
A(t)}\big|a(t)Y_{t}\big|^{2}dt
\Biggr] < +\infty.
\]
\item$\mathcal{H}^{2}(\beta,a)$ is the space of $\mathbb{R}^d$-valued
and $\mathcal{F}_t$-progressively measurable processes $(Z_t)_{t\leq
T}$ such that
\[
\|Z\|_\beta^{2}=\mathbb{E} \Biggl[\int_{0}^{T}e^{\beta
A(t)}|Z_{t}|^{2}dt
\Biggr] < +\infty.
\]
\item$\mathfrak{B}^2$ is the Banach space of the processes $(Y,Z)\in
 (\mathcal{S}^{2}(\beta,a)\cap\mathcal{S}^{2,a}(\beta,a)
)\times\mathcal{H}^{2}(\beta,a)$ with the norm
\[
\big\|(Y,Z)\big\|_\beta=\sqrt{\|aY\|_\beta^{2}+
\|Z\|_\beta^{2}}.
\]
\end{itemize}
We consider the following conditions:
\begin{description}
\item[{$(H1)$} ] The terminal condition $\xi\in\mathcal{L}^{2}(\beta,a)$.
\end{description}
\noindent The coefficient $f : \varOmega\times[0,T] \times\mathbb{R}
\times\mathbb{R}^{d}\longrightarrow\mathbb{R}$ satisfies
\begin{description}
\item[{$(H2)$}] $\forall t\in[0,T]$ $\forall(y,z,y',z')\in\mathbb{R}
\times\mathbb{R}^{d}\times\mathbb{R} \times\mathbb{R}^{d}$, there
are two non-negative $\mathcal{F}_t$-adapted processes $\mu$ and $\gamma
$ such that
\[
\big|f(t,y,z)-f\bigl(t,y',z'\bigr)\big|\leq
\mu(t)\big|y-y'\big|+\gamma(t)\big|z-z'\big|.
\]
\item[{$(H3)$}] There exists $\epsilon>0$ such that $a^2(t):=\mu
(t)+\gamma^2(t) \geq\epsilon$.
\item[{$(H4)$}] For all $(y,z)\in\mathbb{R} \times\mathbb{R}^{d}$,
the process $(f(t,y,z))_t$ is progressively measurable and such that
\[
\frac{f(.,0,0)}{a}\in\mathcal{H}^{2}(\beta,a).
\]
\end{description}
\noindent The two reflecting barriers $L$ and $U$ are two $\mathcal
{F}_t$-adapted and continuous real-valued processes which satisfy
\begin{description}
\item[{$(H5)$}]\ \vspace*{-9pt}
\[
\mathbb{E} \Bigl[\sup_{0 \leq t\leq T}e^{2\beta A(t)}\big|L_{t}^+\big|^{2}
\Bigr] +\mathbb{E} \Bigl[\sup_{0 \leq t\leq T}e^{2\beta A(t)}\big|U_{t}^-\big|^{2}
\Bigr]< +\infty,
\]
where $L^+$ and $U^-$ are the positive and negative parts of $L$ and
$U$, respectively.
\item[{$(H6)$}] $U$ is regular: i.e., there exists a sequence of
$(U^{n})_{n\geq0}$ such that
\begin{description}
\item[$(i)$] $\forall t\leq T$,  $U_{t}^{n}\leq U_{t}^{n+1}$ and
$\lim\limits_{n\rightarrow+\infty}U_{t}^{n}=U_{t}$ $ \mathbb{P}$-a.s
\item[$(ii)$] $\forall n\geq0$, $\forall t\leq T$,
\[
U_{t}^{n}=U_{0}^{n}+\int
_{0}^{t}u_{n}(s)ds+\int
_{0}^{t}v_{n}(s)dB_{s}
\]
where the processes $u_{n}$ and $ v_{n}$ are $\mathcal{F}_{t}$-adapted
such that
\[
\sup_{n\geq0}\sup_{0\leq t\leq T} \bigl(u_{n}(t)
\bigr)^{+}\leq C \quad\mbox {and}\quad\mathbb{E} \Biggl[\int
_{0}^{T}\big|v_{n}(s)\big|^{2}ds
\Biggr]^{\frac
{1}{2}}<+\infty.
\]
\end{description}
\end{description}
%
\begin{defin}
Let $\beta>0$ and $a$ be a non-negative $\mathcal{F}_t$-adapted
process. A solution to DRBSDE is a quadruple $(Y,Z,K^+,K^-)$ satisfying
(\ref{a}) such that
\begin{itemize}
\item$(Y,Z)\in(\mathcal{S}^{2}(\beta,a)\cap\mathcal{S}^{2,a}(\beta
,a))\times\mathcal{H}^{2}(\beta,a)$,
\item$K^\pm\in\mathcal{S}^{2}$ are two continuous and increasing
processes with $K^\pm_0=0$.
\end{itemize}
\end{defin}
\section{A priori estimate}\label{s2}
\begin{lemma}
Let $\beta>0$ be large enough and assume $(H1)-(H6)$ hold. Let
$(Y,Z,K^+,\allowbreak{}K^-)\in(\mathcal{S}^{2}(\beta,a)\cap\mathcal
{S}^{2,a}(\beta,a))\times\mathcal{H}^{2}(\beta,a)\times\mathcal
{S}^{2}\times\mathcal{S}^{2}$ be a solution to DRBSDE with data $(\xi
,f,L,U)$. Then there exists a constant $C_\beta$ depending only on
$\beta$ such that
\begin{align}
& \mathbb{E} \Biggl[\sup_{0\leq t\leq T}e^{\beta A(t)}|Y_{t}|^{2}
+\int_{0}^{T}e^{\beta A(t)}
\bigl(a^2(t)|Y_{t}|^{2}+|Z_{t}|^{2}
\bigr)dt+\big|K_{T}^{+}\big|^2+\big|K_{T}^{-}\big|^2
\Biggr]
\nonumber
\\[-2pt]
&\quad\leq C_\beta\mathbb{E} \Biggl[e^{\beta A(T)}|\xi|^2+
\int_{0}^{T}e^{\beta A(t)}\frac{|f(t,0,0)|^2}{a^2(t)}dt
\nonumber
\\[-2pt]
&\qquad+\sup_{0\leq t\leq T}e^{2\beta
A(t)}
\bigl(\big|L_{t}^+\big|^{2}+\big|U_{t}^-\big|^{2}
\bigr) \Biggr].
\end{align}
\end{lemma}
\begin{proof}
Applying It\^{o}'s formula and Young's inequality, combined with the
stochastic Lipschitz assumption $(H2)$ we can write
\begin{align*}
& e^{\beta A(t)}|Y_t|^2+\int_{t}^{T}
\beta e^{\beta
A(s)}a^2(s)|Y_{s}|^2ds+\int
_{t}^{T}e^{\beta A(s)}|Z_{s}|^2ds
\\[-2pt]
&\quad\leq e^{\beta A(T)}|\xi|^2+\frac{\beta}{2}\int
_{t}^{T}e^{\beta
A(s)}a^2(s)|Y_{s}|^2ds
+\frac{2}{\beta}\int_{t}^{T}e^{\beta A(s)}
\frac
{|f(s,Y_s,Z_s)|^2}{a^2(s)}ds
\\[-2pt]
&\qquad+2\int_{t}^{T}e^{\beta A(s)}Y_{s}dK_{s}^{+}-2
\int_{t}^{T}e^{\beta
A(s)}Y_{s}dK_{s}^{-}-2
\int_{t}^{T}e^{\beta A(s)}Y_{s}Z_{s}dB_{s}
\\[-2pt]
&\quad\leq e^{\beta A(T)}|\xi|^2+\frac{\beta}{2}\int
_{t}^{T}e^{\beta
A(s)}a^2(s)|Y_{s}|^2ds+
\frac{6}{\beta}\int_{t}^{T}e^{\beta
A(s)}a^2(s)|Y_{s}|^2ds
\\[-2pt]
&\qquad+\frac{6}{\beta}\int_{t}^{T}e^{\beta A(s)}|Z_{s}|^2ds+
\frac{6}{\beta
}\int_{t}^{T}e^{\beta A(s)}
\frac{|f(s,0,0)|^2}{a^2(s)}ds
\\
&\qquad+2\int_{t}^{T}e^{\beta A(s)}Y_{s}dK_{s}^{+}-2
\int_{t}^{T}e^{\beta
A(s)}Y_{s}dK_{s}^{-}-2
\int_{t}^{T}e^{\beta A(s)}Y_{s}Z_{s}dB_{s}.
\end{align*}
Using the fact that $dK_{s}^{+}=\mathbh{1}_{\{Y_s=L_s\}}dK_{s}^{+}$ and
$dK_{s}^{-}=\mathbh{1}_{\{Y_s=U_s\}}dK_{s}^{-}$, we have
\begin{align}
\label{e1} & e^{\beta A(t)}|Y_t|^2+\biggl(
\frac{\beta}{2}-\frac{6}{\beta}\biggr)\int_{t}^{T}
e^{\beta A(s)}a^2(s)|Y_{s}|^2ds +\biggl(1-
\frac{6}{\beta}\biggr)\int_{t}^{T}e^{\beta A(s)}|Z_{s}|^2ds
\nonumber
\\[-2pt]
&\quad\leq e^{\beta A(T)}|\xi|^2+\frac{6}{\beta}\int
_{t}^{T}e^{\beta
A(s)}\frac{|f(s,0,0)|^2}{a^2(s)}ds+2\int_{t}^{T}e^{\beta A(s)}L_{s}dK_{s}^{+}
\nonumber
\\[-2pt]
&\qquad-2
\int_{t}^{T}e^{\beta
A(s)}U_{s}dK_{s}^{-}-2
\int_{t}^{T}e^{\beta A(s)}Y_{s}Z_{s}dB_{s}.
\end{align}
Taking expectation on both sides above, we get
\begin{align}
\label{e2} &\mathbb{E} \Biggl[\int_{0}^{T}
e^{\beta A(s)}a^2(s)|Y_{s}|^2ds+\int
_{0}^{T}e^{\beta A(s)}|Z_{s}|^2ds
\Biggr]
\nonumber
\\[-2pt]
&\quad\leq c_\beta\mathbb{E} \Biggl[e^{\beta A(T)}|\xi|^2+
\int_{0}^{T}e^{\beta A(s)}\frac{|f(s,0,0)|^2}{a^2(s)}ds
\nonumber
\\[-2pt]
&\qquad+\sup_{0\leq t\leq T}e^{\beta A(t)} \bigl\llvert
L^+_{t} \bigr\rrvert ^2+ \bigl\llvert K_{T}^{+}
\bigr\rrvert ^2 +\sup_{0\leq t\leq T}e^{\beta A(t)} \bigl
\llvert U^-_{t} \bigr\rrvert ^2+ \bigl\llvert
K_{T}^{-} \bigr\rrvert ^2 \Biggr]
\end{align}
and by the Burkholder--Davis--Gundy's inequality we obtain
\begin{align}
& \mathbb{E}\sup_{0\leq t\leq T}e^{\beta A(t)}|Y_{t}|^{2}
\nonumber
\\[-2pt]
&\quad\leq\mathcal{C}_\beta\mathbb{E} \Biggl[e^{\beta A(T)}|
\xi|^2+\int_{0}^{T}e^{\beta A(s)}
\frac{|f(s,0,0)|^2}{a^2(s)}ds
\nonumber
\\[-2pt]
&\qquad+2\int_{t}^{T}e^{\beta A(s)}L_{s}dK_{s}^{+}-2
\int_{t}^{T}e^{\beta A(s)}L_{s}dK_{s}^{-}
\Biggr]\label{e3}
\\[-2pt]
&\quad\leq\mathcal{C}_\beta\mathbb{E} \Biggl[e^{\beta A(T)}|
\xi|^2+\int_{0}^{T}e^{\beta A(s)}
\frac{|f(s,0,0)|^2}{a^2(s)}ds
\nonumber
\\[-2pt]
&\qquad+\sup_{0\leq t\leq T}e^{2\beta A(t)}
\bigl(\big|L_{t}^+\big|^{2}+\big|U_{t}^-\big|^{2}
\bigr) +\big|{K_{T}^{+}}\big|^2+\big|{K_{T}^{-}}\big|^2
\Biggr].\label{e31}
\end{align}
To conclude, we now give an estimate of ${K_{T}^{+}}^2$ and
${K_{T}^{-}}^2$. From the equation
\[
K_{T}^{+}-K_{T}^{-}=Y_0-
\xi-\int_0^Tf(s,Y_s,Z_s)ds+
\int_0^TZ_sdB_s
\]
and the stochastic Lipschitz property $(H2)$, we have
\begin{align*}
& \mathbb{E} \bigl[\big|K_{T}^{+}-K_{T}^{-}\big|^2
\bigr]
\\[-2pt]
&\quad\leq 4\mathbb{E} \Biggl[\sup_{0\leq t\leq T}e^{\beta
A(t)}|Y_{t}|^{2}+|
\xi|^2+\biggl(1+\frac{3}{\beta}\biggr)\int_0^Te^{\beta
A(s)}|Z_s|^2ds
\\
&\qquad+\frac{3}{\beta}\int_{0}^{T}e^{\beta
A(s)}a^2(s)|Y_{s}|^{2}ds
+\frac{3}{\beta}\int_{0}^{T}e^{\beta A(s)}
\frac
{|f(s,0,0)|^2}{a^2(s)}ds \Biggr].
\end{align*}
Combining this with (\ref{e3}), we derive that
\begin{align}
\label{e4} \mathbb{E} \bigl\llvert K_{T}^{+}
\bigr\rrvert ^2+\mathbb{E} \bigl\llvert K_{T}^{-}
\bigr\rrvert ^2 &\leq\mathfrak{C}_\beta\mathbb{E}
\Biggl[e^{\beta A(T)}|\xi|^{2} +\int_{0}^{T}e^{\beta A(s)}
\frac{|f(s,0,0)|^2}{a^2(s)}ds
\nonumber
\\
&\quad+\sup_{0\leq t\leq T}e^{2\beta
A(t)}
\bigl(\big|L_{t}^+\big|^{2}+\big|U_{t}^-\big|^{2}
\bigr) \Biggr] +\frac{1}{2}\mathbb{E} \bigl\llvert {K_{T}^{+}}
\bigr\rrvert ^2+\frac{1}{2}\mathbb {E} \bigl\llvert
{K_{T}^{-}} \bigr\rrvert ^2.
\end{align}
The desired result is obtained by estimates (\ref{e2}), (\ref{e31}) and
(\ref{e4}).
\end{proof}
\section{Existence and uniqueness of solution}\label{s3}
\subsection{The obstacle $U$ is regular}
In this part, we apply the penalization method and the fixed point
theorem to give the existence of the solution to the DRBSDE (\ref{a}).
We first consider the special case when the generator does not depend
on $(y,z)$:
\[
f(t,y,z)=g(t).
\]
\begin{thm}\label{t}
Assume that $\frac{g}{a}\in\mathcal{H}^{2}(\beta,a)$ and
${(H1)}$--${(H6)}$ hold. Then, the doubly reflected BSDE (\ref{a}) with
data $(\xi,g,L,U)$ has a unique solution $(Y,Z,K^+,K^-)$ that belongs
to $(\mathcal{S}^{2}(\beta,a)\cap\mathcal{S}^{2,a}(\beta,a))\times
\mathcal{H}^{2}(\beta,a)\times\mathcal{S}^{2}\times\mathcal{S}^{2}$.
\end{thm}
For all $n\in\mathbb{N}$, let $(Y^n,Z^n,K^{n+})$ be the $\mathcal
{F}_t$-adapted process with values in $(\mathcal{S}^{2}(\beta,a)\cap
\mathcal{S}^{2,a}(\beta,a))\times\mathcal{H}^{2}(\beta,a)\times
\mathcal{S}^{2}$ being a solution to the reflected BSDE with data $(\xi
, g(t)-n(y-U_t)^+,L)$. That is
\begin{eqnarray}
\label{r1} &&\hspace{-1cm} \displaystyle \left\{ %
\begin{array}{@{}ll}
Y^n_{t}=\xi+\displaystyle\int_{t}^{T} g(s)ds-n\displaystyle\int_{t}^{T}\bigl(Y^n_s-U_s\bigr)^+ds+K^{n+}_{T}-K^{n+}_{t}-\displaystyle\int_{t}^{T}Z^n_{s}dB_{s}& \hbox{}\\
Y^n_{t}\geq L_{t}, \ \forall t\leq T \ \mbox{and}\ \displaystyle
\int_{0}^{T}\bigl(Y^n_{t}-L_{t}\bigr)dK^{n+}_{t}=0.&
\hbox{}
\end{array}
\right.
\end{eqnarray}
We denote $K^{n-}_{t}:=n\int_0^t(Y^n_s-U_s)^+ds$ and
$g^n(s,y):=g(s)-n(y-U_s)^+$.

We have divided the proof of Theorem \ref{t} into sequence of lemmas.
%
\begin{lemma}\label{l1}
There exists a positive constant $C$ such that
\[
\sup_{0\leq t\leq T} n\bigl(Y_{t}^{n}-U_{t}
\bigr)^{+}\leq C\quad\mathbb{P}\hbox{-}a.s.
\]
\end{lemma}
\begin{proof}
For all $n,m \geq0$, let $(Y^{n,m},Z^{n,m})$ be the solution to the
following BSDE
\[
Y_{t}^{n,m} =\xi-\int_{t}^{T}
\bigl\{ g(s)+m\bigl(Y_{s}^{n,m}-L_{s}
\bigr)^{-}-n\bigl(Y_{s}^{n,m}-U_{s}
\bigr)^{+} \bigr\}ds-\int_{t}^{T}Z_{s}^{n,m}dB_{s}.
\]
We denote $\bar{Y}^{n,m}=Y^{n,m}-U^{m}$. Then we have
\begin{align*}
\bar{Y}_{t}^{n,m}&=\xi-U_{T}^{m}+\int
_{t}^{T}\bigl(g(s)+u_{m}(s)\bigr)ds-n
\int_{t}^{T}\bigl(\bar{Y}_{s}^{n,m}-
\bigl(U_{s}-U_{s}^m\bigr)\bigr)^{+}ds
\\
&\quad +m\int_{t}^{T}\bigl(\bar{Y}_{s}^{n,m}-
\bigl(L_{s}-U_{s}^m\bigr)\bigr)^{-}ds-
\int_{t}^{T}\bigl(Z_{s}^{n,m}-v_{n}(s)
\bigr)dB_{s}.
\end{align*}
For $n\geq0$, let $\mathcal{D}_{n}$ be the class of $\mathcal
{F}_t$-progressively measurable process taking values in $[0,n]$. For
$\nu\in\mathcal{D}_{n}$ and $\lambda\in\mathcal{D}_{m}$ we denote
$R_{t}=e^{-\int_{0}^{t}(\nu(s)+\lambda(s))ds}$. Applying It\^{o}'s
formula to $R_{t}\bar{Y}_{t}^{n,m}$ and using the same arguments as on
page 2042 of \cite{CK}, one can show that
\[
\bar{Y}_{t}^{n,m} \leq\operatorname*{ess
\,sup}_{\lambda\in\mathcal{D}_{m}} \operatorname *{ess\,inf}_{\nu\in\mathcal{D}_{n}}\mathbb{E} \Biggl[
\int_{t}^{T}e^{-\int
_{t}^{s}(\nu(r)+\lambda(r))dr}\big|u_{m}(s)\big|ds
|\mathcal{F}_{t} \Biggr].
\]
From the assumption $(H6)(ii)$, we can write $\bar{Y}_{t}^{n,m}\vee0
\leq\frac{C}{n}$. It follows that
\[
\forall t\leq T,\qquad n\bigl(\bar{Y}_{t}^{n,m}\vee0\bigr)
\xrightarrow[m\to +\infty]{} n\bigl(Y_{t}^{n}-U_{t}
\bigr)^{+}\leq C \quad\mathbb{P}\hbox{-}a.s.\qedhere
\]
\end{proof}
\begin{lemma}\label{lll}
There exists a positive constant $C'_\beta$ depending only on $\beta$
such that for all $n\geq0$
\begin{align*}
\label{e5} &\mathbb{E} \Biggl[\sup_{0\leq t\leq T}e^{\beta
A(t)}\big|Y_{t}^{n}\big|^{2}+
\int_{0}^{T}e^{\beta A(t)}a^2(t)\big|Y_{t}^{n}\big|^{2}dt
+\int_{0}^{T}e^{\beta A(t)}\big|Z_{t}^{n}\big|^{2}dt+\big|K_{T}^{n+}\big|^2
\Biggr]
\nonumber
\\
&\quad\leq C'_\beta\mathbb{E} \Biggl[e^{\beta A(T)}|
\xi|^2+\int_{0}^{T}e^{\beta A(t)}
\biggl\llvert \frac{g(t)}{a(t)} \biggr\rrvert ^2dt
\\
&\qquad+\sup_{0\leq t\leq T}e^{2\beta
A(t)}\big|U_{t}^-\big|^{2}+
\sup_{0\leq t\leq T}e^{2\beta
A(t)}\big|L_{t}^+\big|^{2}
\Biggr].
\end{align*}
\end{lemma}
\begin{proof}
It\^{o}'s formula implies for $t\leq T$:
\begin{align*}
&\beta\mathbb{E}\int_{t}^{T}e^{\beta
A(s)}a^2(s)\big|Y_{s}^{n}\big|^{2}ds+
\mathbb{E}\int_{t}^{T}e^{\beta
A(s)}\big|Z_{s}^{n}\big|^{2}ds
\\
&\quad\leq \mathbb{E}e^{\beta A(T)}|\xi|^2 +\frac{\beta}{2}
\mathbb{E}\int_{t}^{T}e^{2\beta A(s)}a^2(s)\big|Y_{s}^{n}\big|^{2}ds
+\frac{2}{\beta}\mathbb{E}\int_{t}^{T}e^{\beta A(s)}
\frac
{|g(s)|^2}{a^2(s)}ds
\\
&\qquad+2\mathbb{E} \Biggl[\sup_{n\geq0}\sup
_{0\leq t\leq
T}n\bigl(Y_t^n-U_t
\bigr)^+\int_{t}^{T}e^{\beta A(s)}U^-_sds
\Biggr] +2\mathbb{E} \Biggl[\int_{t}^{T}e^{\beta A(s)}L_sdK_s^{n+}
\Biggr].
\end{align*}
Here we used the fact that $-nY^n_s(Y^n_s-U_s)^+\leq nU^-(Y^n_s-U_s)^+$
and $dK^{n+}_s=\mathbh{1}_{\{Y^n_s=L_s\}}dK^{n+}_s$.
We conclude, by the Burkholder--Davis--Gundy's inequality, that
\begin{align*}
&\mathbb{E}\sup_{0\leq t\leq T}e^{\beta A(t)}\big|Y^n_t\big|^2+
\mathbb{E}\int_{0}^{T}e^{\beta A(s)}a^2(s)\big|Y_{s}^{n}\big|^{2}ds+
\mathbb {E}\int_{0}^{T}e^{\beta A(s)}\big|Z_{s}^{n}\big|^{2}ds
\\
&\quad\leq c'_p\mathbb{E} \Biggl[e^{\beta A(T)}|
\xi|^2 +\int_{0}^{T}e^{\beta A(s)}
\frac{|g(s)|^2}{a^2(s)}ds
\\
&\qquad+\sup_{0\leq t\leq T}e^{2\beta A(t)}\big|U_t^-\big|^2+
\sup_{0\leq t\leq T}e^{2\beta A(t)}\big|L_{t}^+\big|^{2}+\big|K_{T}^{n+}\big|^2
\Biggr].
\end{align*}
In the same way as (\ref{e4}), we can prove that
\begin{align*}
\mathbb{E}\big|K_{T}^{n+}\big|^2&\leq
\mathcal{C}'_p\mathbb {E} \Biggl[e^{\beta A(T)}|
\xi|^2 +\int_{0}^{T}e^{\beta A(s)}
\frac{|g(s)|^2}{a^2(s)}ds
\\
&\quad+\sup_{0\leq t\leq T}e^{2\beta A(t)}\big|U_t^-\big|^2+
\sup_{0\leq t\leq T}e^{2\beta A(t)}\big|L_{t}^+\big|^{2}
\Biggr].
\end{align*}
We obtain the desired result.
\end{proof}
\begin{lemma}\label{ll}
There exist two $\mathcal{F}_t$-adapted processes $(Y_t)_{t\leq T}$ and
$(K^+_t)_{t\leq T}$ such that $Y^n\searrow Y$, $K^{n+}\nearrow K^+$ and
\[
\mathbb{E} \Bigl[\sup_{0\leq t\leq T}\big|K_{t}^{n+}-K_{t}^{+}\big|^{2}
\Bigr]\xrightarrow[n\to+\infty]{}0.
\]
\end{lemma}
\begin{proof}
The comparison \xch{Theorem \ref{tc}}{theorem \ref{tc}} (below) shows that $Y^0_t\geq Y^n_t\geq
Y_t^{n+1}$ and $K^{n+}_t\leq K^{(n+1)+}_t$ for all $t\leq T$.
Therefore, there exist processes $Y$ and $K^+$ such that, as
$n\rightarrow+\infty$, for all $t\leq T$, $Y^n_t\searrow Y_t$ and
$K^{n+}_t\nearrow K^+_t$. Since the process $K^+$ is continuous, it
follows by Dini's theorem that
\[
\mathbb{E} \Bigl[\sup_{0\leq t\leq T}\big|K_{t}^{n+}-K_{t}^{+}\big|^{2}
\Bigr]\xrightarrow[n\to+\infty]{}0.\qedhere
\]
\end{proof}
\begin{lemma}\label{l}
\[
\mathbb{E} \Bigl[\sup_{0\leq t\leq T}e^{\beta
A(t)}\big|
\bigl(Y_{t}^{n}-U_{t}\bigr)^+\big|^{2}
\Bigr]\xrightarrow[n\to+\infty]{}0.
\]
\end{lemma}
\begin{proof}
Since $Y_t\leq Y_t^n\leq Y_t^0$, we can replace $U_t$ by $U_t\vee Y^0$;
that is, we may assume that
$\mathbb{E}\sup_{0\leq t\leq T}e^{\beta A(t)} \llvert U_{t} \rrvert ^{2}<+\infty$.

Let $(\widetilde{Y}^n,\widetilde{Z}^n,\widetilde{K}^n)$ be the solution
to the following Reflected BSDE associated with $(\xi, g-n(y-U), L)$:
\begin{eqnarray}
\label{d2} &&\hspace{-1cm} \displaystyle \left\{ %
\begin{array}{@{}ll}
\widetilde{Y}^n_{t}=\xi+\displaystyle\int_{t}^{T}\bigl(g(s)-n\bigl(\widetilde
{Y}^n_s-U_s\bigr)\bigr)ds+\widetilde{K}^n_{T}-\widetilde{K}^n_{t}-\displaystyle
\int_{t}^{T}\widetilde{Z}^n_{s}dB_{s}& \hbox{}\\
\widetilde{Y}^n_{t}\geq L_t,\ \forall t\leq T \ \mbox{and}\
\displaystyle\int_{0}^{T}\bigl(\widetilde{Y}^n_{t}-L_{t}\bigr)d\widetilde
{K}^n_{t}=0. & \hbox{}
\end{array}
\right.
\end{eqnarray}
The comparison \xch{Theorem \ref{tc}}{theorem \ref{tc}} shows that $Y^n\leq\widetilde{Y}^n$
and $d\widetilde{K}^n\leq dK^{n+} \leq dK^+$.
Let $\tau\leq T$ be a stopping time. Then we can write
\[
\widetilde{Y}^n_{\tau}=\mathbb{E} \Biggl[e^{-n(T-\tau)}
\xi+\int_{\tau
}^{T}e^{-n(s-\tau)}
\bigl(g(s)+nU_s\bigr)ds+\int_{\tau}^{T}e^{-n(s-\tau
)}d
\widetilde{K}^n_{s}|\mathcal{F}_{\tau} \Biggr].
\]
Since $\mathbb{E}\sup_{0\leq t\leq T}e^{\beta A(t)}U_{t}^{2}<+\infty$,
we obtain
\[
e^{-n(T-\tau)}\xi+n\int_{\tau}^{T}e^{-n(s-\tau)}U_{s}ds
\xrightarrow[n\to+\infty]{} \xi\mathbh{1}_{\tau=T}+U_{\tau}\mathbh
{1}_{\tau<T} \quad\mathbb{P}\hbox{-}\mathrm{a.s.} \; \mathrm{in} \;
\mathcal{L}^2
\]
and the conditional expectation converges also in $\mathcal{L}^2$. Moreover,
\[
\Biggl\llvert \int_{\tau}^{T}e^{-n(s-\tau)} g(s)ds
\Biggr\rrvert ^2 \leq\int_{\tau}^{T}e^{\beta A(s)}
\biggl\llvert \frac{g(s)}{a(s)} \biggr\rrvert ^2ds\int
_{\tau}^{T}e^{-2n(s-\tau)}e^{-\beta A(s)}a^2(s)ds.
\]
Then
\[
\int_{\tau}^{T}e^{-n(s-\tau)} g(s)ds\xrightarrow[n
\to+\infty]{}0 \quad \mathbb{P}\hbox{-}\mbox{a.s. in }\mathcal{L}^2.
\]
In addition,
\[
0\leq\int_{\tau}^{T}e^{-n(s-\tau)}d
\widetilde{K}_{s}^{n} \leq\int_{\tau}^{T}e^{-n(s-\tau)}dK^+_{s}
\xrightarrow[n\to+\infty ]{}0\mbox{ in }\mathcal{L}^1.
\]
Consequently,
\[
\widetilde{Y}^n_{\tau}\xrightarrow[n\to+\infty]{} \xi
\mathbh{1}_{\tau=T}+U_{\tau}\mathbh{1}_{\tau<T} \quad\mathbb
{P}\mbox{-a.s. in }\mathcal{L}^1.
\]
Therefore, $Y_\tau\leq U_\tau$ $\mathbb{P}$-a.s. We deduce, from
Theorem 86 page 220 in Dellacherie and Meyer \cite{DM}, that $Y_t\leq
U_t$ for all $t\leq T$ $\mathbb{P}$-a.s and then $e^{\beta
A(t)}(Y^n_t-U_t)^+\searrow0$ for all $t\leq T$ $\mathbb{P}$-a.s.
By Dini's theorem, we have $\sup_{0\leq t\leq T}e^{\beta
A(t)}(Y^n_t-U_t)^+\searrow0$ $\mathbb{P}\mbox{-a.s.}$
and the result follows from the Lebesgue's dominated convergence theorem.
\end{proof}
\begin{lemma}
There exist two processes $(Z_t)_{t\leq T}$ and $(K^-_t)_{t\leq T}$ such that
\[
\mathbb{E}\int_{0}^{T}e^{\beta A(t)}a^2(t)\big|Y_{t}^{n}-Y_{t}\big|^2dt
+\mathbb{E}\int_{0}^{T}e^{\beta A(t)}\big|Z_{t}^{n}-Z_{t}\big|^{2}dt
\xrightarrow[n\to+\infty]{}0.
\]
Moreover,
\[
\mathbb{E}\sup_{0\leq t\leq T}e^{\beta A(t)}\big|Y_{t}^{n}-Y_{t}\big|^{2}
+\mathbb{E}\sup_{0\leq t\leq T}\big|K_{t}^{n-}-K_{t}^{-}\big|^{2}
\xrightarrow[n\to+\infty]{}0.
\]
\end{lemma}
\begin{proof}
For all $n\geq p\geq0$ and $t\leq T$, applying It\^{o}'s formula and
taking expectation yields that
\begin{align*}
&\mathbb{E} \Biggl[e^{\beta A(t)}\big|Y_{t}^{n}\,{-}\,Y_{t}^p\big|^{2}
\,{+}\,\beta\int_{t}^{T}e^{\beta A(s)}a^2(s)\big|Y_{s}^{n}\,{-}\,Y_{s}^p\big|^2ds\,{+}
\int_{t}^{T}e^{\beta A(s)}\big|Z_{s}^{n}\,{-}\,Z_{s}^p\big|^2ds
\Biggr]
\\
&\quad\leq2\mathbb{E} \Biggl[\int_{t}^{T}e^{\beta
A(s)}
\bigl(Y_{s}^{p}-U_{s}\bigr)^+n
\bigl(Y_{s}^{n}- U_{s}\bigr)^+ds \Biggr]
\\
&\qquad+2\mathbb{E} \Biggl[\int_{t}^{T}e^{\beta
A(s)}
\bigl(Y_{s}^{n}-U_{s}\bigr)^+p
\bigl(Y_{s}^{p}- U_{s}\bigr)^+ds \Biggr]
\\
&\quad\leq\mathbb{E} \Bigl[\sup_{0\leq t\leq T}
\bigl(e^{\beta
A(t)}\bigl(Y_{t}^{p}-U_{t}
\bigr)^{+}\bigr)^{2} \Bigr]^{\frac{1}{2}} \mathbb{E}
\Biggl[ \Biggl(\int_{t}^{T}n\bigl(Y_{s}^{n}-U_{s}
\bigr)^{+}ds \Biggr)^{2} \Biggr]^{\frac{1}{2}}
\\
&\qquad+\mathbb{E} \Bigl[\sup_{0\leq t\leq T}
\bigl(e^{\beta
A(t)}\bigl(Y_{t}^{n}-U_{t}
\bigr)^{+}\bigr)^{2} \Bigr]^{\frac{1}{2}} \mathbb{E}
\Biggl[ \Biggl(\int_{t}^{T}p\bigl(Y_{s}^{p}-U_{s}
\bigr)^{+}ds \Biggr)^{2} \Biggr]^{\frac{1}{2}}
\end{align*}
since $(Y_{s}^{n}-Y_{s}^p)d(K_{s}^{n+}- K_{s}^{p+})\leq0$.
Therefore, using Lemmas \ref{l1} and \ref{l}, we obtain
\[
\mathbb{E}\int_{0}^{T}e^{\beta A(s)}a^2(s)\big|Y_{s}^{n}-Y_{s}^p\big|^2ds
+\mathbb{E}\int_{0}^{T}e^{\beta A(s)}\big|Z_{s}^{n}-Z_{s}^p\big|^{2}ds
\xrightarrow[n,p\to+\infty]{}0.
\]
It follows that $(Z^n)_{n\geq0}$ is a Cauchy sequence in complete
space $\mathcal{H}^2(\beta,a)$.
Then there exists an $\mathcal{F}_t$-progressively measurable process
$(Z_t)_{t\leq T}$ such that the sequence $(Z^n)_{n\geq0}$ tends toward
$Z$ in $\mathcal{H}^2(\beta,a)$.
On the other hand, by the Burkholder--Davis--Gundy's inequality, one
can derive that
\begin{align*}
&\mathbb{E}\sup_{0\leq t\leq T}e^{\beta
A(t)}\big|Y_{t}^{n}-Y_{t}^p\big|^{2}
\\
&\quad\leq\mathbb{E} \Bigl[\sup_{0\leq t\leq T}\bigl(e^{\beta
A(t)}
\bigl(Y_{t}^{p}-U_{t}\bigr)^{+}
\bigr)^{2} \Bigr]^{\frac{1}{2}} \mathbb{E} \Biggl[ \Biggl(\int
_{t}^{T}n\bigl(Y_{s}^{n}-U_{s}
\bigr)^{+}ds \Biggr)^{2} \Biggr]^{\frac{1}{2}}
\\
&\qquad+\mathbb{E} \Bigl[\sup_{0\leq t\leq T}
\bigl(e^{\beta
A(t)}\bigl(Y_{t}^{n}-U_{t}
\bigr)^{+}\bigr)^{2} \Bigr]^{\frac{1}{2}} \mathbb{E}
\Biggl[ \Biggl(\int_{t}^{T}p\bigl(Y_{s}^{p}-U_{s}
\bigr)^{+}ds \Biggr)^{2} \Biggr]^{\frac{1}{2}}
\\
&\qquad+\frac{1}{2}\mathbb{E}\sup_{0\leq t\leq T}e^{\beta
A(t)}\big|Y_{t}^{n}-Y_{t}^p\big|^{2}
+2c^2\mathbb{E}\int_{t}^{T}e^{\beta A(s)}\big|Z_{s}^{n}-Z_{s}^p\big|^2ds
\end{align*}
where $c$ is a universal 
non-negative constant. It follows that
\[
\mathbb{E}\sup_{0\leq t\leq T}e^{\beta
A(t)}\big|Y_{t}^{n}-Y_{t}^p\big|^{2}
\xrightarrow[n,p\to+\infty]{}0
\]
and then
\[
\mathbb{E} \Biggl[\sup_{0\leq t\leq T}e^{\beta A(t)}\big|Y_{t}^{n}-Y_{t}\big|^{2}
+\int_{0}^{T}e^{\beta A(t)}a^2(t)\big|Y_{t}^{n}-Y_{t}\big|^2dt
\Biggr]\xrightarrow[n\to+\infty]{}0.
\]
Now, we set
\[
K^-_t=Y_t-Y_0+\int_0^tg(s)ds+K^+_t-K^+_0-
\int_0^tZ_sdB_s.
\]
One can show, at least for a subsequence (which we still index by $n$), that
\[
\mathbb{E}\sup_{0\leq t\leq T}\big|K_{t}^{n-}-K_{t}^{-}\big|^{2}
\xrightarrow [n\to+\infty]{}0.
\]
The proof is completed.
\end{proof}
\begin{proof}[Proof of Theorem \ref{t}]
Obviously, the process $(Y_t,Z_t,K^+_t,K^-_t)_{t\leq T}$ satisfies, for all
$t\leq T$,
\[
Y_{t}=\xi+\int_{t}^{T} g(s)ds+
\bigl(K_{T}^{+}-K_{t}^{+}\bigr)-
\bigl(K_{T}^{-}-K_{t}^{-}\bigr)-\int
_{t}^{T}Z_{s}dB_{s}.
\]
Since $Y^n_t\geq L_t$ and from Lemma \ref{l} we have $L_t\leq Y_t \leq U_t$.

In the following, we want to show that
\[
\int_{0}^{T}(Y_{t}-L_{t})dK_{t}^{+}=
\int_{0}^{T}(U_{t}-Y_{t})dK_{t}^{-}=0
\quad\mathbb{P}\mbox{-a.s.}
\]
Note that
\[
\int_{0}^{T}(Y_{t}-L_{t})dK_{t}^{+}=
\int_{0}^{T}\bigl(Y_{t}-Y_{t}^n
\bigr)dK_{t}^{+}+\int_{0}^{T}
\bigl(Y_{t}^n-L_{t}\bigr) \bigl(dK_{t}^{+}-dK_{t}^{n+}
\bigr).
\]
Let $\omega\in\varOmega$ be fixed. It follows from Lemma \ref{ll} that,
for any $\varepsilon>0$, there exists $n(\omega)$ such that $\forall
n\geq n(\omega)$, $Y_t(\omega)\leq Y_t^n(\omega)+\varepsilon$. Hence
\begin{equation}
\label{f} \int_{0}^{T}\bigl(Y_{t}(
\omega)-Y_{t}^n(\omega)\bigr)dK^{+}_{t}(
\omega)\leq \varepsilon K^{+}_{T}(\omega).
\end{equation}
On the other hand, since the function $(Y_{t}(\omega)-L_{t}(\omega))_{t\leq T}$
is continuous, then there exists a sequence of non-negative step
functions $(f^m(\omega))_{m\geq0}$ which converges uniformly on
$[0,T]$ to $Y_{t}(\omega)-L_{t}(\omega)$. That is
\[
\big|Y_{t}(\omega)-L_{t}(\omega)-f_t^m(
\omega)\big|<\varepsilon.
\]
It follows that
\begin{align*}
&\int_{0}^{T}\bigl(Y_{t}(
\omega)-L_{t}(\omega)\bigr)d \bigl(K_{t}^{+}(
\omega )-K_{t}^{n+}(\omega) \bigr)
\\
&\quad\leq\varepsilon\bigl(K_{T}^{+}(\omega)+K_{T}^{n+}(
\omega)\bigr)+\int_{0}^{T}f_t^m(
\omega)d \bigl(K_{t}^{+}(\omega)-K_{t}^{n+}(
\omega) \bigr).
\end{align*}
Further,
\[
\varepsilon\bigl(K_{T}^{+}(\omega)+K_{T}^{n+}(
\omega)\bigr)\xrightarrow[n\to +\infty]{}2\varepsilon K_T^+(\omega)
\]
and, since $(f^m(\omega))_{m\geq0}$ is a step function,
\[
\int_{0}^{T}f_t^m(
\omega)d\bigl(K_{t}^{+}(\omega)-K_{t}^{n+}(
\omega )\bigr)\xrightarrow[m\to+\infty]{}0.
\]
Therefore, we have
\[
\limsup_{n\to+\infty}\int_{0}^{T}
\bigl(Y_{t}^n-L_{t}\bigr)d\bigl(K_{t}^{+}-K_{t}^{n+}
\bigr)\leq2\varepsilon K_T^+(\omega).
\]
From (\ref{f}) we deduce that
\[
\int_{0}^{T}(Y_{t}-L_{t})dK_{t}^{+}
\leq3\varepsilon K_T^+(\omega).
\]
The arbitrariness of $\varepsilon$ and $Y\geq L$, show that $\int_{0}^{T}(Y_{t}-L_{t})dK_{t}^{+}=0$.
Further, by Lemma \ref{ll} and the result treated on p. 465 of Saisho
\cite{YS} we can write
\begin{equation}
\label{e7} \int_{0}^{T}\bigl(U_{s}-Y^n_{s}
\bigr)n\bigl(Y^n_{s}-U_{s}\bigr)ds\xrightarrow[n
\to+\infty ]{}\int_{0}^{T}(U_{s}-Y_{s})dK^{-}_{s}.
\end{equation}
Since $\int_{0}^{T}(U_{s}-Y^n_{s})n(Y^n_{s}-U_{s})ds=\int_{0}^{T}(U_{s}-Y^n_{s})dK^{n-}_{s}\leq0$ for each $n\geq0$ $\mathbb{P}$-a.s.
and for each $n,m\geq0$, $n\neq m$,
\[
\mathbb{E} \Biggl[ \Biggl\llvert \int_0^T
\bigl(Y_{s}^{n}-Y_{s}^m
\bigr)dK^{m-}_s \Biggr\rrvert \Biggr] \leq\mathbb{E} \Bigl[
\sup_{0\leq t\leq T}e^{\beta
A(t)}\big|Y_{t}^{n}-Y_{t}^m\big|K^{m-}_T
\Bigr]\xrightarrow[n,m\to+\infty]{}0.
\]
Then we have
\begin{equation}
\label{e8} \limsup_{n\rightarrow+\infty}\int_0^T
\bigl(U_{s}-Y^n_{s}\bigr)dK^{n-}_t
\leq 0\quad\mathbb{P}\mbox{-a.s.}
\end{equation}
Combining (\ref{e7}) and (\ref{e8}), we get $\int_0^T(U_{s}-Y_{s})dK^{-}_s\leq0\ \mathbb{P}\mbox{-a.s.}$ Noting that
$Y\leq U$, we conclude that $\int_0^T(U_{s}-Y_{s})dK^{-}_s=0$.
Consequently, $(Y_t,Z_t,K^+_t,K^-_t)$ is the solution to (\ref{a})
associated to the data $(\xi, g, L, U)$.
\end{proof}

We can now state the main result:
\begin{thm}\label{fpt}
Assume $(H1)$--$(H6)$ hold for a sufficient large $\beta$. Then
DRBSDE~(\ref{a}) has a unique solution $(Y,Z,K^+,K^-)$ that belongs to
$(\mathcal{S}^{2}(\beta,a)\cap\mathcal{S}^{2,a}(\beta,a))\times
\mathcal{H}^{2}(\beta,a)\times\mathcal{S}^{2}\times\mathcal{S}^{2}$.
\end{thm}
\begin{proof}
Given $(\phi,\psi)\in\mathfrak{B}^2$, consider the following DRBSDE :
\begin{align}
\label{b1}  \displaystyle \left\{ %
\begin{array}{@{}l}
Y_{t}=\xi\,{+}\displaystyle\int_{t}^{T} f(s,\phi_{s},\psi
_{s})ds\,{+}\,(K_{T}^{+}-K_{t}^{+})-(K_{T}^{-}-K_{t}^{-})\,{-}\displaystyle\int_{t}^{T}Z_{s}dB_{s} \quad t\,{\leq}\, T
\\
L_t\leq{Y}_{t}\leq U_t,\ \forall t\leq T \ \mbox{and}\ \displaystyle\int_{0}^{T}({Y}_{t}-L_{t})d{K}^+_{t}=\displaystyle\int_{0}^{T}(U_{t}-Y_t)d{K}^-_{t}=0.
\end{array}
\right.
\end{align}
From ${(H2)}$ and ${(H3)}$, we have
\[
\big|f(t,\phi_{t},\psi_{t})\big|^2 \leq3
\bigl(a(t)^4|\phi_t|^2+a(t)^2|
\psi_t|^2+\big|f(t,0,0)\big|^2 \bigr).
\]
It follows from ${(H4)}$ that $\frac{f}{a}\in\mathcal{H}^{2}(\beta,a)$
and then (\ref{b1}) has a unique solution $(Y,Z,\allowbreak{}K^{+},K^{-})$.

We define a mapping
\[
\begin{array}{ccccc}
\varphi&: & \mathfrak{B}^2 &\longrightarrow& \mathfrak{B}^2 \\
& & (\phi,\psi)&\longmapsto& (Y,Z)\\
\end{array} %
\]
Let $\varphi(\phi,\psi)=(Y,Z)$ and $\varphi(\phi',\psi')=(Y',Z')$ where
$(Y,Z,K^{+},K^{-})$ (resp. $(Y',\allowbreak{}Z',K^{+'},K^{-'})$) is the
unique solution to the DRBSDE associated with data $(\xi,\break
{}f(.,\phi,\psi),L,U)$ (resp. $(\xi,f(.,\phi',\psi'),L,U)$). Denote
$\Delta{\varGamma}=\varGamma-\varGamma'$ for $\varGamma=\break Y,Z,K^+,K^-,\phi,\psi$ and
$\Delta f_t=f(t,{\phi'}_{t},{\psi'}_{t})-f(t,\phi_{t},\psi_{t})$.
Applying It\^{o}'s formula to $e^{\beta A(t)}|\Delta Y_t|^2$ and taking
expectation we have
\begin{align*}
&\mathbb{E}e^{\beta A(t)}|\Delta Y_t|^2 +\beta
\mathbb{E}\int_{t}^{T} e^{\beta A(s)}a^2(s)|
\Delta Y_{s}|^2ds +\mathbb{E}\int_{t}^{T}e^{\beta A(s)}|
\Delta Z_{s}|^2ds
\\
&\quad\leq2\mathbb{E}\int_{t}^{T}e^{\beta A(s)}
\Delta Y_{s} \Delta f_s ds
\\
&\quad\leq\alpha\beta\mathbb{E}\int_{t}^{T}
e^{\beta
A(s)}a^2(s)|\Delta Y_{s}|^2ds +
\frac{2}{\alpha\beta}\mathbb{E}\int_{t}^{T}
e^{\beta
A(s)}\bigl(a^2(s)|\Delta\phi_s|^2+|
\Delta\psi_s|^2\bigr)ds.
\end{align*}
We have used the fact that $\Delta Y_{s}d(\Delta K_{s}^{+}-\Delta
K_{s}^{-})\leq0$. Choosing $\alpha\beta=4$ and $\beta>5$, we can write
\[
\bigl\llVert \varphi(\phi,\psi) \bigr\rrVert _\beta^2 \leq
\frac{1}{2} \bigl\llVert (\phi ,\psi) \bigr\rrVert _\beta^2.
\]
It follows that $\varphi$ is a strict contraction mapping on $\mathfrak
{B}^2$ and then $\varphi$ has a unique fixed point which is the
solution to the DRBSDE (\ref{a}).
\end{proof}
\begin{remark}
If we consider $U=+\infty$, we obtain the BSDE with one continuous
reflecting barrier $L$, then we proved the existence and uniqueness of
the solution to RBSDE (\ref{RBSDE}) by means of a penalization method.
Before this work, Wen L\"{u} \cite{WL} showed the existence and
uniqueness result for this class of equations via the Snell envelope notion.\looseness=-1
\end{remark}
%
\subsection{Completely separated barriers}
In this section we will prove the existence of solution to (\ref{a})
when the barriers are completely separated, i.e., $ L_t<U_t$, $\forall
t\leq T$. Then
\begin{itemize}
\item[$(\mathcal{H}7)$] there exists a continuous semimartingale
\[
H_t=H_0+\int_0^th_sdB_s-V^+_t+V^-_t,
\quad H_T=\xi
\]
with $h\in\mathcal{H}^2(0,a)$ and $V^\pm\in\mathcal{S}^2$ ($V^\pm_0=0$)
are two nondecreasing continuous processes, such that
\begin{equation}
\label{h} L_t\leq H_t\leq U_t \qquad0
\leq t\leq T.
\end{equation}
\end{itemize}
We will show the existence by the general penalization method. We first
consider the special case when the generator does not depend on $(y,z)$:
\[
f(t,y,z)=f(t).
\]
Let $(Y^{n},Z^{n})\in(\mathcal{S}^{2}(\beta,a)\cap\mathcal
{S}^{2,a}(\beta,a))\times\mathcal{H}^{2}(\beta,a)$ be solution to
the following BSDE
\begin{align}
\label{b} Y^{n}_{t}&=\xi+\int_{t}^{T}f(s)ds-n
\int_{t}^{T}\bigl(Y^{n}_s-U_s
\bigr)^+ds+n\int_{t}^{T}\bigl(Y^{n}_s-L_s
\bigr)^-ds
\nonumber
\\
&\quad-\int_{t}^{T}Z^{n}_{s}dB_{s}.
\end{align}
We denote $K^{n+}_t:=n\int_{0}^{t}(Y^{n}_s-L_s)^-ds$, $K^{n-}_t:=n\int_{0}^{t}(Y^{n}_s-U_s)^+ds$, $K^{n}_t=K^{n+}_t-K^{n-}_t$ and
$f^{n}(s,y)=f(s)-n(y-U_s)^++n(y-L_s)^-$.

Now let us derive the uniform a priori estimates of
$(Y^{n},Z^{n},K^{n+},K^{n-})$.
\begin{lemma}\label{l2}
There exists a positive constant $\kappa$ independent of $n$ such that,
$\forall n\geq0$,
\begin{align*}
&\mathbb{E} \Biggl[\sup_{0\leq t\leq T}e^{\beta
A(t)}\big|Y_{t}^{n}\big|^{2}+
\int_{0}^{T}e^{\beta
A(t)}a^2(t)\big|Y_{t}^{n}\big|^{2}dt
\\
&\quad+\int_{0}^{T}e^{\beta
A(t)}\big|Z_{t}^{n}\big|^{2}dt+\big|K_{T}^{n+}\big|^2+\big|K_{T}^{n-}\big|^2
\Biggr]\leq\kappa.
\end{align*}
\end{lemma}
\begin{proof}
Consider the RBSDE with data $(\xi,f,L)$. That is,
\begin{eqnarray}
\label{c} && \displaystyle \left\{ %
\begin{array}{@{}ll}
\overline{Y}_{t}=\xi+\displaystyle\int_{t}^{T}f(s)ds+\overline
{K}_{T}-\overline{K}_{t}-\displaystyle\int_{t}^{T}\overline
{Z}_{s}dB_{s}& \hbox{}\\
\overline{Y}_{t}\geq L_t,\ \forall t\leq T \ \mbox{and}\
\displaystyle\int_{0}^{T}(\overline{Y}_{t}-L_{t})d\overline{K}_{t}=0. &
\hbox{}
\end{array}
\right.
\end{eqnarray}
From Appendix~\ref{apdx} there exists a unique triplet of processes
$(\overline{Y},\overline{Z},\overline{K})\in(\mathcal{S}^{2}(\beta
,a)\cap\mathcal{S}^{2,a}(\beta,a))\times\mathcal{H}^{2}(\beta
,a)\times\mathcal{S}^{2}$ being the solution to RBSDE (\ref{c}). We
consider the penalization equation associated with the RBSDE (\ref{c}),
for $n\in\mathbb{N}$,
\begin{equation*}
\overline{Y}^n_{t}=\xi+\int_{t}^{T}
f(s)ds+n\int_{t}^{T}\bigl(\overline
{Y}^n_s-L_s\bigr)^-ds-\int
_{t}^{T}\overline{Z}^n_{s}dB_{s}.
\end{equation*}
The Remark \ref{Rmqcomp} implies that $\overline{Y}^0_{t}\leq\overline
{Y}^n_{t}\leq\overline{Y}^{n+1}$ and $Y^n_{t}\leq\overline{Y}^n_{t}$
for all $t\leq T$. Therefore, as $n\longrightarrow+\infty$ for all
$t\leq T$, $\overline{Y}^n_t\nearrow\overline{Y}_t$. Hence
$Y^{n}_{t}\leq\overline{Y}_{t}$.

Similarly, we consider the RBSDE with data $(\xi,f,U)$. There exists a
unique triplet of processes $(\underline{Y},\underline{Z},\underline
{K})\in(\mathcal{S}^{2}(\beta,a)\cap\mathcal{S}^{2,a}(\beta
,a))\times\mathcal{H}^{2}(\beta,a)\times\mathcal{S}^{2}$, which satisfies
\begin{eqnarray}
\label{c1} && \displaystyle \left\{ %
\begin{array}{@{}ll}
\underline{Y}_{t}=\xi+\displaystyle\int_{t}^{T}f(s)ds-(\underline
{K}_{T}-\underline{K}_{t})-\displaystyle\int_{t}^{T}\underline
{Z}_{s}dB_{s}& \hbox{}\\
\underline{Y}_{t}\leq U_t,\ \forall t\leq T \ \mbox{and}\
\displaystyle\int_{0}^{T}(U_t-\underline{Y}_{t})d\underline{K}_{t}=0. &
\hbox{}
\end{array}
\right.
\end{eqnarray}
By the penalization equation associated with the RBSDE (\ref{c1})
\begin{equation*}
\underline{Y}^n_{t}=\xi+\int_{t}^{T}
f(s)ds-n\int_{t}^{T}\bigl(\underline
{Y}^n_s-U_s\bigr)^+ds-\int
_{t}^{T}\underline{Z}^n_{s}dB_{s}
\end{equation*}
and the Remark \ref{Rmqcomp}, we deduce that $Y^{n}_{t}\geq\underline
{Y}_{t}$ for all $t\leq T$. Then we can write
\begin{equation}
\label{m} \mathbb{E}\sup_{0\leq t\leq T}e^{\beta A(t)}\big|Y^{n}_t\big|^2
\leq\max \Bigl\{\mathbb{E}\sup_{0\leq t\leq T}e^{\beta A(t)}|\overline
{Y}_t|^2,\mathbb{E}\sup_{0\leq t\leq T}e^{\beta A(t)}|
\underline {Y}_t|^2 \Bigr\} \leq\kappa.
\end{equation}
On the other hand, using It\^{o}'s formula and taking expectation
implies for $t\leq T$:
\begin{align*}
&\beta\mathbb{E}\int_{t}^{T}e^{\beta
A(s)}a^2(s)\big|Y_{s}^{n}\big|^{2}ds+
\mathbb{E}\int_{t}^{T}e^{\beta
A(s)}\big|Z_{s}^{n}\big|^{2}ds
\\
&\quad\leq\mathbb{E}e^{\beta A(T)}|\xi|^2+2\mathbb{E}\int
_{t}^{T}e^{\beta
A(s)}Y_{s}^{n}f(s)ds
\\
&\qquad-2n\mathbb{E}\int_{t}^{T}e^{\beta
A(s)}Y_{s}^{n}
\bigl(Y_s^{n}-U_s\bigr)^+ds+2n\mathbb{E}\int
_{t}^{T}e^{\beta
A(s)}Y_{s}^{n}
\bigl(Y_s^{n}-L_s\bigr)^-ds
\\
&\quad\leq\mathbb{E}e^{\beta A(T)}|\xi|^2+\frac{\beta}{2}\mathbb{E}
\int_{t}^{T}e^{\beta A(s)}a^2(s)\big|Y_{s}^{n}\big|^2ds+
\frac{2}{\beta}\mathbb {E}\int_{t}^{T}e^{\beta A(s)}
\frac{|f(s)|^2}{a^2(s)}ds
\\
&\qquad+2n\mathbb{E}\int_{t}^{T}e^{\beta
A(s)}U_{s}^{-}
\bigl(Y_s^{n}-U_s\bigr)^+ds+2n\mathbb{E}\int
_{t}^{T}e^{\beta
A(s)}L_{s}^{+}
\bigl(Y_s^{n}-L_s\bigr)^-ds.
\end{align*}
Hence
\begin{align}
\label{ekk} &\frac{\beta}{2}\mathbb{E}\int_{t}^{T}e^{\beta
A(s)}a^2(s)\big|Y_{s}^{n}\big|^{2}ds+
\mathbb{E}\int_{t}^{T}e^{\beta
A(s)}\big|Z_{s}^{n}\big|^{2}ds
\nonumber
\\
&\quad\leq\mathbb{E}e^{\beta A(T)}|\xi|^2+
\frac{2}{\beta
}\mathbb{E}\int_{t}^{T}e^{\beta A(s)}
\biggl\llvert \frac{f(s)}{a(s)} \biggr\rrvert ^2ds+
\frac{1}{\alpha}\mathbb{E}\sup_{0\leq t\leq T}e^{2\beta
A(t)}
\bigl(\big|L_{t}^{+}\big|^2+\big|U_{t}^{-}\big|^2
\bigr)
\nonumber
\\
&\qquad+\alpha\mathbb{E} \Biggl[\int_{t}^{T}n
\bigl(Y_s^{n}-U_s\bigr)^+ds
\Biggr]^2+\alpha\mathbb{E} \Biggl[\int_{t}^{T}n
\bigl(Y_s^{n}-L_s\bigr)^-ds
\Biggr]^2.
\end{align}
Now we need to estimate $\mathbb{E} [\int_{t}^{T}n(Y_s^{n}-U_s)^+ds ]^2+ \mathbb{E} [\int_{t}^{T}n(Y_s^{n}-L_s)^-ds ]^2$. For this, let us consider the
following stopping times
\[
\displaystyle \left\{ %
\begin{array}{@{}ll}
\tau_{0}=0 , & \hbox{} \\
\tau_{2l+1}=\inf\big\{t>\tau_{2l}\;\; |\;\; Y_t^{n}\leq L_t\big\}\wedge T, &
l\geq0 \\
\tau_{2l+2}=\inf\big\{t>\tau_{2l+1}\;\; |\;\; Y_t^{n}\geq U_t\big\}\wedge T, &
l\geq0.
\end{array} \right.
\]
Since $Y$, $L$ and $U$ are continuous processes and $L<U$, $\tau_l<\tau
_{l+1}$ on the set $\{\tau_{l+1}<T\}$.
In addition the sequence $(\tau_l)_{l\geq 0}$ is of stationary type (i.e.
$\forall\omega\in\varOmega$, there exists $l_0(\omega)$ such that $\tau
_{l_0}(\omega)=T$). Indeed, let us set $G=\{\omega\in\varOmega, \tau
_{l}(\omega)<T, \ l\geq0\}$, and we will show that $\mathbb{P}(G)=0$.
We assume that $\mathbb{P}(G)>0$, therefore for $\omega\in G$, we have
$Y_{\tau_{2l+1}}\leq L_{\tau_{2l+1}}$ and $Y_{\tau_{2l}}\geq U_{\tau_{2l}}$.
Since $(\tau_l)_{l\geq 0}$ is nondecreasing sequence then $\tau_l\nearrow\tau$,
hence $U_\tau\leq Y_\tau\leq L_\tau$ which is contradiction since
$L<U$. We deduce that $\mathbb{P}(G)=0$. Obviously $Y^{n}\geq L$ on the
interval $[\tau_{2l},\tau_{2l+1}]$, then the BSDE (\ref{b}) becomes
\begin{equation}
\label{eee1} Y^{n}_{\tau_{2l}}=Y^{n}_{\tau_{2l+1}} +
\int_{\tau_{2l}}^{\tau
_{2l+1}}f(s)ds-n\int_{\tau_{2l}}^{\tau_{2l+1}}
\bigl(Y^{n}_s-U_s\bigr)^+ds-\int
_{\tau_{2l}}^{\tau_{2l+1}}Z^{n}_{s}dB_{s}.
\end{equation}
On the other hand, using the assumption $(\mathcal{H}7)$, we get
\begin{align*}
Y_{\tau_{2l}}^{n}&\geq H_{\tau_{2l}} \mbox{ on } \{\tau_{2l}< T\}\quad\mbox
{and}\quad Y_{\tau_{2l}}^{n}=H_{\tau_{2l}}=\xi\mbox{ on } \{\tau_{2l}=T\}
,\\
Y_{\tau_{2l+1}}^{n}&\leq H_{\tau_{2l+1}} \mbox{ on } \{\tau_{2l+1}< T\}
\quad\mbox{and}\quad Y_{\tau_{2l+1}}^{n}=H_{\tau_{2l+1}}=\xi\mbox{ on }\{\tau
_{2l+1}=T\}.
\end{align*}
From (\ref{eee1}) and the definition of process $H$ we obtain
\begin{align*}
n\int_{\tau_{2l}}^{\tau_{2l+1}}\bigl(Y_s^{n}-U_s
\bigr)^+ds
&\leq H_{\tau_{2l+1}}-H_{\tau_{2l}}+\int_{\tau_{2l}}^{\tau_{2l+1}}f(s)ds
-\int_{\tau_{2l}}^{\tau_{2l+1}}Z_s^{n}dB_s
\\
&\leq\int_{\tau_{2l}}^{\tau_{2l+1}}\bigl(h_s-Z_s^{n}
\bigr)dB_s+\int_{\tau
_{2l}}^{\tau_{2l+1}}\big|f(s)\big|ds+V^-_{\tau_{2l+1}}-V^-_{\tau_{2l}}.
\end{align*}
By summing in $l$, using the fact that $Y^{n}\leq U$ on the interval
$[\tau_{2l+1},\tau_{2l+2}]$, we can write for $t\leq T$
\begin{align}
\label{c2} \mathbb{E} \Biggl[n\int_{t}^{T}
\bigl(Y_s^{n}-U_s\bigr)^+ds
\Biggr]^2&\leq 4 \Biggl(\mathbb{E}\int_{t}^{T}|h_s|^2ds+
\mathbb{E}\int_{t}^{T}e^{\beta
A_s}\big|Z_s^{n}\big|^2ds
\nonumber
\\
&\quad+\frac{T}{\beta}\mathbb{E}\int_{t}^{T}e^{\beta
A_s}
\frac{|f(s)|^2}{a^2(s)}ds+\mathbb{E}\big|V^-_T\big|^2 \Biggr).
\end{align}
In the same way, we obtain
\begin{align}
\label{c3} \mathbb{E} \Biggl[n\int_{t}^{T}
\bigl(Y_s^{n}-L_s\bigr)^-ds
\Biggr]^2&\leq 4 \Biggl(\mathbb{E}\int_{t}^{T}|h_s|^2ds+
\mathbb{E}\int_{t}^{T}e^{\beta
A_s}\big|Z_s^{n}\big|^2ds
\nonumber
\\
&\quad+\frac{T}{\beta}\mathbb{E}\int_{t}^{T}e^{\beta
A_s}
\frac{|f(s)|^2}{a^2(s)}ds+\mathbb{E}\big|V^+_T\big|^2 \Biggr).
\end{align}
Combining (\ref{c2}), (\ref{c3}) with (\ref{ekk}), we obtain the
desired result.
\end{proof}
\begin{lemma}\label{l3}\quad
\begin{enumerate}
\item\quad$\mathbb{E}\sup\limits_{0\leq t\leq T}e^{\beta
A(t)}|(Y_{t}^{n}-U_{t})^+|^{2}\xrightarrow[n\to+\infty]{}0$.
\item\quad$\mathbb{E}\sup\limits_{0\leq t\leq T}e^{\beta
A(t)}|(Y_{t}^{n}-L_{t})^-|^{2}\xrightarrow[n\to+\infty]{}0$.
\end{enumerate}
\end{lemma}
\begin{proof}
Consider the following BSDE for each $n\in\mathbb{N}$
\begin{align*}
\widehat{Y}^{n}_{t}&=\xi+\int_{t}^{T}f(s)ds+n
\int_{t}^{T}\bigl(L_s-
\widehat{Y}^{n}_s\bigr)ds -\int_{t}^{T}
\widehat{Z}^{n}_{s}dB_{s}
\\
&=\xi+\int_{t}^{T}f(s)ds+n\int
_{t}^{T}\bigl(\widehat{Y}^{n}_s-L_s
\bigr)^-ds -n\int_{t}^{T}\bigl(L_s-
\widehat{Y}^{n}_s\bigr)^-ds\,{-}\int_{t}^{T}
\widehat{Z}^{n}_{s}dB_{s}.
\end{align*}
By the Remark \ref{Rmqcomp}, we have $Y^{n}_{t}\geq\widehat{Y}^{n}_{t}$
for all $t\leq T$. Let $\nu$ be a stopping time such that $\nu\leq T$. Then
\begin{equation}
\label{ek4} \widehat{Y}^n_{\nu}=\mathbb{E}
\Biggl[e^{-n(T-\nu)}\xi+\int_{\nu
}^{T}e^{-n(s-\nu)}f(s)ds
+n\int_{\nu}^{T}e^{-n(s-\nu)}L_sds|
\mathcal{F}_{\nu} \Biggr].
\end{equation}
It is easily seen that
\[
e^{-n(T-\nu)}\xi+n\int_{\nu}^{T}e^{-n(s-\nu)}L_{s}ds
\xrightarrow[n\to +\infty]{} \xi\mathbh{1}_{\nu=T}+L_{\nu}
\mathbh{1}_{\nu<T} \qquad \mathbb{P}\mbox{-a.s.\ in\ }
\mathcal{L}^2.
\]
Moreover, the conditional expectation converges also in $\mathcal
{L}^2$. In addition, by the H\"{o}lder inequality, we have
\begin{align*}
& \Biggl\llvert \int_{\nu}^{T}e^{-n(s-\nu)}
f(s)ds \Biggr\rrvert ^2
\\
&\quad\leq \Biggl(\int_{\nu}^{T}e^{\beta A(s)} \biggl
\llvert \frac{f(s)}{a(s)} \biggr\rrvert ^2ds \Biggr) \Biggl(\int
_{\nu}^{T}e^{-2n(s-\nu)-\beta A(s)}a^2(s)ds
\Biggr)\xrightarrow[n\to+\infty]{}0.
\end{align*}
Thus $\int_{\nu}^{T}e^{-n(s-\nu)} f(s)ds\xrightarrow[n\to+\infty]{}0$ %
$\mathbb{P}\mbox{-a.s. \ in \ }\mathcal{L}^2$.

Now, we denote
\begin{align*}
&\widehat{y}^n_t:=e^{-n(T-t)}\xi+\int
_{t}^{T}e^{-n(s-t)}\bigl(f(s)+nL_s
\bigr)ds,\\
&\tilde{y}^{n}_t:=e^{-n(T-t)}L_T+\int
_{t}^{T}e^{-n(s-t)}\bigl(f(s)+nL_s
\bigr)ds
\end{align*}
and
\[
X^n_t:=e^{-n(T-t)}L_T+n\int
_{t}^{T}e^{-n(s-t)}L_sds-L_t.
\]
By the fact that $L$ is uniformly continuous on $[0,T]$, it can be
shown that the sequence
$(X^n_t)_{n\geq1}$ uniformly converges in $t$, and the same for
$(X^{n-}_t)_{n\geq1}$.
Lebesgue's dominated convergence theorem implies that
\begin{align*}
&\lim\limits
_{n\rightarrow+\infty}\mathbb{E}\sup_{0\leq t\leq
T}e^{\beta A(t)}\big|
\bigl(\widehat{y}^{n}_t-L_t
\bigr)^-\big|^2 =\lim_{n\rightarrow+\infty}\mathbb{E}\sup
_{0\leq t\leq T}e^{\beta
A(t)}\big|\bigl(\tilde{y}^{n}_t-L_t
\bigr)^-\big|^2
\\
&\quad\leq2\lim_{n\rightarrow+\infty}\mathbb{E} \Biggl[\sup
_{0\leq t\leq T}e^{\beta A(t)}\big|X^{n-}_t\big|^2
+\sup_{0\leq t\leq T}e^{\beta A(t)} \Biggl\llvert \int
_{t}^{T}e^{-n(s-t)}f(s)ds \Biggr\rrvert
^2 \Biggr]=0.
\end{align*}
So, from (\ref{ek4}), Jensen's inequality and Doob's maximal quadratic
inequality (see Theorem 20, p. 11 in \cite{Pro}), we have
\begin{align*}
\mathbb{E}\sup_{0\leq t\leq T}e^{\beta A(t)}\big|\bigl(
\widehat{Y}^{n}_t-L_t\bigr)^-\big|^2 &
\leq\mathbb{E}\sup_{0\leq t\leq T}e^{\beta A(t)} \bigl\llvert
\mathbb{E} \bigl[\bigl(\widehat{y}^{n}_t-L_t
\bigr)^-|\mathcal{F}_t \bigr] \bigr\rrvert ^2
\\
&\leq4\mathbb{E}\sup_{0\leq t\leq T}e^{\beta A(t)} \bigl\llvert
\bigl(\widehat {y}^{n}_t-L_t\bigr)^- \bigr
\rrvert ^2\xrightarrow[n\to+\infty]{}0.
\end{align*}
From the fact that $Y^{n}_{t}\geq\widehat{Y}^{n}_{t}$ for all $t\leq T$
we deduce that
\[
\lim\limits
_{n\rightarrow+\infty}\mathbb{E}\sup_{0\leq t\leq T}e^{\beta
A(t)}\big|
\bigl(Y^{n}_t-L_t\bigr)^-\big|^2=0.
\]
Similarly to proof of the Lemma \ref{l}, we can obtain
\[
\lim\limits
_{n\rightarrow+\infty}\mathbb{E}\sup_{0\leq t\leq T}e^{\beta
A(t)}\big|
\bigl(Y^{n}_t-U_t\bigr)^+\big|^2=0.\qedhere
\]
\end{proof}
\begin{lemma}\label{l4}
For each $n\geq p\geq0$, we have
\begin{align*}
&\mathbb{E} \Biggl[\sup_{0\leq t\leq T}e^{\beta
A(t)}\big|Y_{t}^{n}-Y_{t}^{p}\big|^{2}+
\int_{0}^{T}e^{\beta
A(t)}a^2(t)\big|Y_{t}^{n}-Y_{t}^p\big|^{2}dt
\\
&\quad+\int_{0}^{T}e^{\beta
A(t)}\big|Z_{t}^{n}-Z_{t}^p\big|^{2}dt+
\sup_{0\leq t\leq
T}\big|K_{t}^{n}-K_{t}^{p}\big|^{2}
\Biggr]\xrightarrow[n,p\to+\infty]{}0.
\end{align*}
\end{lemma}
\begin{proof}
It\^{o}'s formula implies that
\begin{align*}
&\mathbb{E}e^{\beta A(t)}\big|Y_{t}^{n}-Y_{t}^p\big|^{2}+
\mathbb{E}\int_{t}^{T}e^{\beta A(s)}\bigl(\beta
a^2(s)\big|Y_{s}^{n}-Y_{s}^p\big|^2+\big|Z_{s}^{n}-Z_{s}^p\big|^2
\bigr)ds
\\
&\quad\leq2\mathbb{E}\int_{t}^{T}e^{\beta
A(s)}
\bigl(Y_{s}^{n}-Y_{s}^p\bigr)
\bigl(dK_s^{n+}-dK_s^{p+}\bigr)
\\
&\qquad-2\mathbb{E}\int_{t}^{T}e^{\beta
A(s)}
\bigl(Y_{s}^{n}-Y_{s}^p\bigr)
\bigl(dK_s^{n-}-dK_s^{p-}\bigr)
\\
&\quad\leq2\mathbb{E}\int_{t}^{T}e^{\beta A(s)}
\bigl(Y_{s}^{n}-L_{s}\bigr)^-dK_s^{p+}
+2\mathbb{E}\int_{t}^{T}e^{\beta A(s)}
\bigl(Y_{s}^{p}-L_{s}\bigr)^-dK_s^{n+}
\\
&\qquad+2\mathbb{E}\int_{t}^{T}e^{\beta A(s)}
\bigl(Y_{s}^{n}-U_{s}\bigr)^+dK_s^{p-}
+2\mathbb{E}\int_{t}^{T}e^{\beta A(s)}
\bigl(Y_{s}^{p}-U_{s}\bigr)^+dK_s^{n-}.
\end{align*}
Hence
\begin{align*}
&\beta\mathbb{E}\int_{t}^{T}e^{\beta A(s)}a^2(s)\big|Y_{s}^{n}-Y_{s}^p\big|^2ds
+\mathbb{E}\int_{t}^{T}e^{\beta A(s)}\big|Z_{s}^{n}-Z_{s}^p\big|^2ds
\\
&\quad\leq2\mathbb{E}\sup_{0\leq t\leq T}e^{\beta A(t)}
\bigl(Y_{t}^{n}-L_{t}\bigr)^-K_T^{p+}
+2\mathbb{E}\sup_{0\leq t\leq T}e^{\beta
A(t)}\bigl(Y_{t}^{p}-L_{t}
\bigr)^-K_T^{n+}
\\
&\qquad+2\mathbb{E}\sup_{0\leq t\leq T}e^{\beta A(t)}\bigl(Y_{t}^{n}-U_{t}
\bigr)^+K_T^{p-} +2\mathbb{E}\sup_{0\leq t\leq T}e^{\beta A(t)}
\bigl(Y_{t}^{p}-U_{t}\bigr)^+K_T^{n-}.
\end{align*}
Lemma \ref{l3} implies that
\begin{equation}
\mathbb{E}\int_{t}^{T}e^{\beta A(s)}a^2(s)\big|Y_{s}^{n}-Y_{s}^p\big|^2ds
+\mathbb{E}\int_{t}^{T}e^{\beta
A(s)}\big|Z_{s}^{n}-Z_{s}^p\big|^2)ds
\xrightarrow[n,p\to+\infty]{}0.
\end{equation}
On the other hand, by the Burkholder--Davis--Gundy's inequality, we get
\begin{equation}
\mathbb{E}\sup_{0\leq t\leq T}e^{\beta
A(t)}\big|Y_{t}^{n}-Y_{t}^p\big|^2
\xrightarrow[n,p\to+\infty]{}0.
\end{equation}
From the equation
\begin{equation}
K^n_t=Y^n_0-Y^n_t-
\int_0^tf(s)ds+\int_0^tZ^n_sdB_s
\quad0\leq t\leq T,
\end{equation}
we can conclude that
\begin{equation}
\mathbb{E}\sup_{0\leq t\leq T}\big|K_{t}^{n}-K_{t}^p\big|^2
\xrightarrow[n,p\to +\infty]{}0.
\end{equation}
The proof is completed.
\end{proof}
The main result of this section is the following:
\begin{thm}\label{tt}
Assume that $L<U$. Then the DRBSDE (\ref{a}) has a unique solution
$(Y,Z,K^+,K^-)$ that belongs to $(\mathcal{S}^{2}(\beta,a)\cap\mathcal
{S}^{2,a}(\beta,a))\times\mathcal{H}^{2}(\beta,a)\times\mathcal
{S}^{2}\times\mathcal{S}^{2}$.
\end{thm}
\begin{proof}
From Lemma \ref{l4}, we obtain that there exists an adapted process
$(Y,Z,K)\in(\mathcal{S}^{2}(\beta,a)\cap\mathcal{S}^{2,a}(\beta
,a))\times\mathcal{H}^{2}(\beta,a)\times\mathcal{S}^{2}$ such that
\begin{align}
\label{c4} &\mathbb{E} \Biggl[\sup_{0\leq t\leq T}e^{\beta
A(t)}\big|Y_{t}^{n}-Y_{t}\big|^{2}+
\int_{0}^{T}e^{\beta
A(t)}a^2(t)\big|Y_{t}^{n}-Y_{t}\big|^{2}dt
\\
&\quad+\int_{0}^{T}e^{\beta
A(t)}\big|Z_{t}^{n}-Z_{t}\big|^{2}dt+
\sup_{0\leq t\leq
T}\big|K_{t}^{n}-K_{t}\big|^{2}
\Biggr]\xrightarrow[n\to+\infty]{}0
\nonumber
.
\end{align}
Then, passing to the limit as $n\rightarrow+\infty$ in the equation
\[
Y^n_t=\xi+\int_t^Tf(s)ds+K^n_T-K^n_t-
\int_t^TZ^n_sdB_s,
\]
we obtain
\[
Y_t=\xi+\int_t^Tf(s)ds+K_T-K_t-
\int_t^TZ_sdB_s.
\]
Let $\tau\leq T$ be a stopping time, by Lemma \ref{l2} we obtain that
the sequences $K^{n\pm}_\tau$ are bounded in $\mathcal{L}^2$,
consequently, there exist $\mathcal{F}_\tau$-measurable random
variables $K^{\pm}_\tau$ in $\mathcal{L}^2$, such that there exist the
subsequences of $K^{n\pm}_\tau$ weakly converging in $K^{\pm}_\tau$.

Now we set $\mathcal{K}_\tau=K^+_\tau-K^-_\tau$. By \cite{Yosida}
(Mazu's Lemma, p. 120), there exists, for every $n\in\mathbb{N}$, an
integer $N\geq n$ and a convex combination $\sum_{j=n}^N\zeta_j^{\tau
,n}(K^{\pm}_\tau)_j$ with $\zeta_j^{\tau,n}\geq0$ and $\sum_{j=n}^N\zeta_j^{\tau,n}=1$ such that
\begin{equation}
\label{c6} \mathcal{K}_\tau^{n\pm}:=\sum
_{j=n}^N\zeta_j^{\tau,n}
\bigl(K^{\pm}_\tau \bigr)_j\xrightarrow[n\to+
\infty]{}K^{\pm}_\tau.
\end{equation}
Denoting $\mathcal{K}_\tau^{n}=\mathcal{K}_\tau^{n+}-\mathcal{K}_\tau
^{n-}$, it follows that
\begin{equation}
\label{c5} \mathbb{E}\big|\mathcal{K}_\tau^{n}-
\mathcal{K}_\tau\big|^2\xrightarrow[n\to +\infty]{}0.
\end{equation}
Thanks to (\ref{c4}), we have $\|K_\tau^{n}-K_\tau\|_{\mathcal
{L}^2}<\varepsilon$ for all $\varepsilon>0$. Therefore
\begin{align*}
\big\|\mathcal{K}_\tau^{n}-K_\tau
\big\|_{\mathcal{L}^2}&=\Bigg\|\sum_{j=n}^N\zeta
_j^{\tau,n}\bigl(\bigl(K^{\pm}_\tau
\bigr)_j-K_\tau\bigr)\Bigg\|_{\mathcal{L}^2}
\\
&\leq \sum_{j=n}^N
\zeta_j^{\tau,n}\big\|\bigl(K^{\pm}_\tau
\bigr)_j-K_\tau\big\|_{\mathcal
{L}^2}<\varepsilon.
\end{align*}
Hence
\begin{equation}
\label{c5new} \mathbb{E}\big|\mathcal{K}_\tau^{n}-K_\tau\big|^2
\xrightarrow[n\to+\infty]{}0.
\end{equation}
Combining (\ref{c5}) and (\ref{c5new}), we obtain $\mathcal{K}_\tau=K_\tau
$ a.s. Therefore, from Theorem 86, p. 220 in \cite{DM} we have $\mathcal
{K}_t=K_t$ for all $t\leq T$. On the other hand, (\ref{c6}) implies
that, for $\tau=T$, there exists a subsequence of $\mathcal
{K}_T^{n+}:=\sum_{j=n}^N\zeta_j^{T,n}(K^{+}_T)_j$ (resp. $\mathcal
{K}_T^{n-}:=\sum_{j=n}^N\zeta_j^{T,n}(K^{-}_T)_j$) converging a.s. to
$K^+_T$ (resp. $K^-_T$). Then for $\mathbb{P}$-a.s. $\omega\in\varOmega$,
the sequence $\mathcal{K}^{n+}_T(\omega)$ (resp. $\mathcal
{K}^{n-}_T(\omega)$) is bounded. Using Theorem 4.3.3, p. 88 in \cite
{Chung}, there exists a subsequence of $\mathcal{K}_t^{n+}(\omega)$
(resp. $\mathcal{K}_t^{n-}(\omega)$) tending to $K^{+}_t(\omega)$
(resp. $K^{-}_t(\omega)$), weakly.

On the other hand, by the definition of stopping times $(\tau_l)_{l\geq
0}$, we have
\[
\left\{ %
\begin{array}{@{}ll}
Y_t^{n}> L_t , & \hbox{on $[\tau_{2l},\tau_{2l+1}[$;} \\
Y_t^{n}< U_t , & \hbox{on $[\tau_{2l+1},\tau_{2l+2}[$.}
\end{array} \right.
\]
Then
\[
L_t \mathbh{1}_{[\tau_{2i},\tau_{2i+1}]}(t)\leq Y_t^{n}
\leq U_t \mathbh {1}_{[\tau_{2i+1},\tau_{2i+2}]}(t).
\]
By summing in $i$, $i=0,\ldots, l$ and passing to limit in $n$, we obtain
$L_t\leq Y_t\leq U_t$.
Now, we would have to show the Skorokhod's conditions. Indeed, since
$\mathcal{K}_t^{n+}(\omega)$ tends to $K^{+}_t(\omega)$, using the
result treated in p. 465 of \cite{YS} we can write
\begin{equation}
\label{c7} \int_{0}^{T}\bigl(Y^n_{t}(
\omega)-L_t(\omega)\bigr)d\mathcal{K}_t^{n+}(
\omega )\xrightarrow[n\to+\infty]{}\int_{0}^{T}
\bigl(Y_{t}(\omega)-L_t(\omega )\bigr)dK^{+}_{t}(
\omega).
\end{equation}
Since $\int_{0}^{T}(Y^n_{t}-L_t)dK^{n+}_{t}\leq0$, $\forall n\geq0$ a.s.,
and $\forall n,m\geq0$, $n\neq m$,
\[
\mathbb{E} \Biggl[ \Biggl\llvert \int_0^T
\bigl(Y_{t}^{n}-Y_{t}^m
\bigr)dK^{m+}_t \Biggr\rrvert \Biggr] \leq\mathbb{E} \Bigl[
\sup_{0\leq t\leq T}e^{\beta
A(t)}\big|Y_{t}^{n}-Y_{t}^m\big|K^{m+}_T
\Bigr]\xrightarrow[n,m\to+\infty]{}0,
\]
then by
\[
\int_0^T\bigl(Y_{t}^{n}-L_{t}
\bigr)dK^{m+}_t=\int_0^T
\bigl(Y_{t}^{n}-Y_{t}^m
\bigr)dK^{m+}_t+\int_0^T
\bigl(Y_{t}^{m}-L_{t}\bigr)dK^{m+}_t
\]
we have
\begin{equation}
\label{c8} \limsup_{n\rightarrow+\infty}\int_0^T
\bigl(Y^n_{t}-L_t\bigr)d\mathcal
{K}^{n+}_t\leq0\quad\mathbb{P}\mbox{-a.s.}
\end{equation}
Combining (\ref{c7}) and (\ref{c8}), we get $\int_0^T(Y_{t}-L_t)dK^{+}_t\leq0\
\mathbb{P}\mbox{-a.s.}$ Noting that
$Y\geq L$, we conclude that $\int_0^T(Y_{t}-L_t)dK^{+}_t=0$. By a
similar 
consideration, we can prove $\int_0^T(U_{t}-Y_t)dK^{-}_t=0$.

Finally, using the fixed point theorem 
we construct a strict contraction mapping $\varphi$ on $\mathfrak{B}^2$
and 
conclude that $(Y_t,Z_t,K^+_t,K^-_t)$ is the unique solution to DRBSDE
(\ref{a}) associated with data $(\xi,f,L,U)$.
\end{proof}
%
\section{Comparison theorem}\label{s4}
In this section we prove a comparison theorem for the DRBSDE under the
stochastic Lipschitz assumptions on generators.
\begin{thm}\label{tcc}
Let $(Y^1,Z^1,K^{1+},K^{1-})$ and $(Y^2,Z^2,K^{2+},K^{2-})$ be
respectively the solutions to the DRBSDE with data $(\xi
^1,f^1,L^1,U^1)$ and $(\xi^2,f^2,L^2,U^2)$. Assume in addition the following:
\begin{itemize}
\item$\xi^1\leq\xi^2$ a.s.
\item$f^1(t,Y^2,Z^2)\leq f^2(t,Y^2,Z^2)$ \quad$\forall t\in[0,T]$ a.s.
\item$L^1_{t}\leq L^2_{t}$ and $U^1_{t}\leq U^2_{t}$\quad$\forall
t\in[0,T]$ a.s.
\end{itemize}
Then
\[
\forall t\leq T,\qquad Y^1_{t}\leq Y^2_{t}
\quad\mathit{a.s.}
\]
\end{thm}
\begin{proof}
Let $\bar{\Re}=\Re^1-\Re^2$ for $\Re=Y,Z,K^+,K^+,\xi$ and
\begin{itemize}
\item$\zeta_t=\mathbh{1}_{\{\bar{Y}_{t}\neq0\}}\displaystyle\frac
{f^1(t,Y^1_{t},Z^1_{t})-f^1(t,Y^2_{t},Z^1_{t})}{\bar{Y}_{t}}$;
\item$\eta_t=\mathbh{1}_{\{\bar{Z}_{t}\neq0\}}\displaystyle\frac
{f^1(t,Y^2_{t},Z^1_{t})-f^1(t,Y^2_{t},Z^2_{t})}{\bar{Z}_{t}}$;
\item$\delta_t=f^1(t,Y^2_{t},Z^2_{t})-f^2(t,Y^2_{t},Z^2_{t})$.
\end{itemize}
Applying the Meyer--It\^{o} formula (Theorem 66, p. 210 in \cite{Pro}),
there exists a continuous nondecreasing process $(\mathcal{A}_t)_{t\leq
T}$ such that
\begin{align*}
\big|\bar{Y}_{t}^{+}\big|^{2} &=2\int
_{t}^{T}\bar{Y}_{s}^{+} (
\zeta_s\bar{Y}_{s}+\eta_s\bar
{Z}_{s}+\delta_s )ds -2\int_{t}^{T}
\bar{Y}_{s}^{+}\bar{Z}_{s}dB_{s}
\\
&\quad+2\int_{t}^{T}\bar{Y}_{s}^{+}d
\bar{K}^+_{s}-2\int_{t}^{T}\bar
{Y}_{s}^{+}d\bar{K}^-_{s}-(\mathcal{A}_T-
\mathcal{A}_t).
\end{align*}
Suppose in addition that
\[
\mathbb{E}\int_0^T\mu_tdt<+\infty
\quad\mbox{ and }\quad \mathbb{E}\int_0^T|
\gamma_t|^2dt<+\infty.
\]
Let $\{\varGamma_{t,s}, 0\leq t\leq s\leq T\}$ be the process defined as
\[
\varGamma_{t,s}=\exp \Biggl\{ \int_{t}^{s}
\biggl(\zeta_{u}-\frac{1}{2}|\eta _{u}|^{2}
\biggr)du+\int_{t}^{s}\eta_{u}dB_{u}
\Biggr\}>0
\]
being a solution to the linear stochastic differential equation
\[
\varGamma_{t,s}=1+ \int_{t}^{s}
\zeta_{u}\varGamma_{t,u}du+\int_{t}^{s}
\eta _{u}\varGamma_{t,u}dB_{u}.
\]
%
Applying the integration by parts and taking expectation yield
\begin{align*}
&\mathbb{E} \bigl[e^{\beta A(t)}\big|\bar{Y}_{t}^{+}\big|^{2}
\bigr] +\beta\mathbb{E}\int_{0}^{T}e^{\beta A(s)}
\varGamma_{t,s}a^2(s)\big|\bar {Y}_{s}^{+}\big|^2ds
\\
&\quad\leq\mathbb{E} \Biggl[\int_{t}^{T}e^{\beta A(s)}
\varGamma_{t,s}\zeta _s{\big|\bar{Y}_{s}^+\big|}^2
ds \Biggr] +2\mathbb{E} \Biggl[\int_{t}^{T}e^{\beta A(s)}
\varGamma_{t,s}\delta_s\bar {Y}_{s}^{+}ds
\Biggr]
\\
&\qquad+2\mathbb{E}\int_t^Te^{\beta A(s)}
\varGamma_{t,s}\bar{Y}_{s}^{+}dK^+_s
-2\mathbb{E}\int_t^Te^{\beta A(s)}
\varGamma_{t,s}\bar{Y}_{s}^{+}dK^-_s.
\end{align*}
Remark that
\[
\bar{Y}_{s}^{+}d\bar{K}^+_s=
\bigl(L^1_s-Y^2_s\bigr)\mathbh
{1}_{Y^1_s>Y^2_s}dK^{1+}_s-\bigl(Y^1_s-L^2_s
\bigr)\mathbh {1}_{Y^1_s>Y^2_s}dK^{2+}_s\leq0
\]
and
\[
\bar{Y}_{s}^{+}d\bar{K}^-_s=
\bigl(Y^1_s-U^2_s\bigr)\mathbh
{1}_{Y^1_s>Y^2_s}dK^{2-}_s-\bigl(U^1_s-Y^2_s
\bigr)\mathbh {1}_{Y^1_s>Y^2_s}dK^{1-}_s\leq0.
\]
Since $\delta_s\leq0$ and $|\zeta_s|\leq a^2(s)$, one can derive that
\[
\mathbb{E} \bigl[e^{\beta A(t)}\big|\bar{Y}_{t}^{+}\big|^{2}
\bigr]\leq0.
\]
It follows that $\bar{Y}_{t}^{+}=0$, i.e $Y^1_{t}\leq Y^2_{t}$ for all
$t\leq T$ a.s.
\end{proof}
\begin{remark}\label{Rmqcomp}
\mbox{}
\begin{itemize}
\item If $U^i=+\infty$ for $i=1,2$, then $dK^{i-}=0$ and the comparison
holds also for the reflected BSDE (\ref{RBSDE}).
\item If $U^i=+\infty$ and $L^i=-\infty$ for $i=1,2$, then $dK^{i\pm
}=0$ and the comparison holds also for the BSDE (\ref{BSDE}).
\end{itemize}
\end{remark}
%
\begin{appendix}
\section{Appendix}\label{apdx}
In this section, we study a special case of the reflected BSDE when the
generator depends only on $y$.

We consider the following reflected BSDE
\begin{eqnarray}
\label{r} && \displaystyle \left\{ %
\begin{array}{@{}ll}
Y_{t}=\xi+\displaystyle\int_{t}^{T}
f(s,Y_{s})ds+K_{T}-K_{t}-\displaystyle\int_{t}^{T}Z_{s}dB_{s} & \hbox{}
\\
Y_{t}\geq L_{t} \ \forall t\leq T \ \mbox{ and }\ \displaystyle\int_{0}^{T}(Y_{t}-L_{t})dK_{t}=0 & \hbox{}
\end{array}    \right.
\end{eqnarray}
where $(\xi,f,L)$ satisfies the following assumptions:
\begin{itemize}
\item$\xi\in\mathcal{S}^2(\beta,a)$;
\item$f$ is Lipschitz, i.e. there exists a positive constant $\mu$
such that $\forall(t,y,y')\in[0,T]\times\mathbb{R}\times\mathbb{R}$
\[
\big|f(t,y)-f\bigl(t,y'\bigr)\big|\leq\mu\big|y-y'\big|;
\]
\item$\displaystyle\frac{f(t,0)}{a}\in\mathcal{H}^2(\beta,a)$;
\item$\mathbb{E} [\sup\limits_{0\leq t\leq T}e^{2\beta
A(t)}|L_t^+|^2 ]<+\infty$.
\end{itemize}
As in \cite{EKPPQ}, we prove the existence and uniqueness of a solution
to (\ref{r}) by means of the penalization method. Indeed, for each $n\in
\mathbb{N}$, we consider the following BSDE:
\begin{equation}
\label{rr} Y_{t}^{n}=\xi+\int_{t}^{T}
f\bigl(s,Y_{s}^{n}\bigr)ds+n\int_{t}^{T}
\bigl(Y^n_s-L_s\bigr)^-ds-\int
_{t}^{T}Z_{s}^{n}dB_{s}.
\end{equation}
We denote $K_{t}^{n}:=n\int_{0}^{t}(Y_{s}^{n}-L_{s})^{-}ds$ and
$f^n(t,y)=f(t,y)+n(y-L_t)^-$.
Remark that $f^n$ is Lipschitz and
\begin{align*}
\mathbb{E}|\xi|^{2}+\mathbb{E}\int_0^T
\bigl\llvert f^n(t,0) \bigr\rrvert ^{2}dt
&\leq\mathbb{E} \bigl[e^{\beta A(T)}|\xi|^{2}
\bigr] +\frac{2}{\beta}\mathbb{E} \Biggl[\int_0^Te^{\beta A(t)}
\biggl\llvert \frac
{f(t,0)}{a(t)} \biggr\rrvert ^{2}dt \Biggr]\\
&\quad+2n^2T\mathbb{E} \Bigl[\sup_{0\leq t\leq T}e^{2\beta A(t)}\big|L_t^+\big|^2
\Bigr].
\end{align*}
From \cite{PP}, there exists a unique process $(Y^n,Z^n)$ being a
solution to the BSDE (\ref{rr}). The sequence $(Y^n,Z^n,K^n)_n$
satisfies the uniform estimate
\begin{align*}
&\mathbb{E}\sup_{0\leq t\leq T}e^{\beta A(t)}\big|Y_{t}^{n}\big|^{2}
+\mathbb{E} \Biggl[\int_{0}^{T}e^{\beta A(s)}a^2(s)\big|Y_{s}^{n}\big|^{2}ds
+\mathbb{E}\int_{0}^{T}e^{\beta A(s)}\big|Z_{s}^{n}\big|^{2}ds
\Biggr]
\\
&\quad\leq C\mathbb{E} \Biggl[e^{\beta A(T)}|\xi|^2+\int
_{0}^{T}e^{\beta
A(s)}\frac{|f(s,0)|^2}{a^2(s)}ds +\sup
_{0\leq t\leq T}e^{2\beta A(s)}\big|L_{s}^+\big|^2
\Biggr].
\end{align*}
where $C$ is a positive constant depending only on $\beta$, $\mu$ and
$\epsilon$.

Now we establish the convergence of sequence $(Y^n,Z^n,K^n)$ to the
solution to (\ref{r}). Obviously $f^n(t,y)\leq f^{n+1}(t,y)$ for each
$n\in\mathbb{N}$, and it follows from Remark \ref{Rmqcomp} that
$Y^n\leq Y^{n+1}$. Hence there exists a process $Y$ such that
$Y^n_t\nearrow Y_t$ $0\leq t\leq T$ a.s. From the a priori estimates
and Fatou's lemma, we have
\begin{equation*}
\mathbb{E} \Bigl[\sup_{0\leq t\leq T}e^{\beta A(t)}|Y_{t}|^{2}
\Bigr]\leq \liminf_{n\rightarrow+\infty}\mathbb{E} \Bigl[\sup
_{0\leq t\leq
T}e^{\beta A(t)}\big|Y_{t}^{n}\big|^{2}
\Bigr] \leq C.
\end{equation*}
Then by the dominated convergence, one can derive that
\[
\mathbb{E} \Biggl[\int_{0}^{T}e^{\beta A(s)}\big|Y_{s}^{n}-Y_{s}\big|^2ds
\Biggr]\xrightarrow[n\to+\infty]{}0.
\]
On the other hand, for all $n\geq p\geq0$ and $t\leq T$, we have
\begin{align*}
\label{ee1} &\mathbb{E}e^{\beta A(t)}\big|Y_{t}^{n}-Y_{t}^p\big|^{2}
+\biggl(\beta-\frac{2\mu}{\epsilon}\biggr)\mathbb{E}\int_{t}^{T}e^{\beta
A(s)}a^2(s)\big|Y_{s}^{n}-Y_{s}^p\big|^2ds
\\
&\quad+\mathbb{E}\int_{t}^{T}e^{\beta
A(s)}\big|Z_{s}^{n}-Z_{s}^p\big|^2ds
\nonumber
\\
&\qquad\leq2\mathbb{E}\int_{t}^{T}e^{\beta A(s)}
\bigl(Y_{s}^{n}-L_{s}\bigr)^-dK_{s}^{p}
+\mathbb{E}\int_{t}^{T}e^{\beta A(s)}
\bigl(Y_{s}^{p}-L_{s}\bigr)^-dK_{s}^{n}.
\end{align*}
%
Similarly to Lemma \ref{l3}, we can easily prove that
\begin{equation}
\label{e11} \mathbb{E}\sup_{0\leq t\leq T}e^{\beta
A(t)}\big|
\bigl(Y_{t}^{n}-L_{t}\bigr)^-\big|^{2}
\xrightarrow[n\to+\infty]{}0.
\end{equation}
By the above result an the a priori estimates, one can derive that
\[
\mathbb{E} \Biggl[\int_{t}^{T}e^{\beta A(s)}
\bigl(Y_{s}^{n}-L_{s}\bigr)^-dK_{s}^{p}
+\int_{t}^{T}e^{\beta A(s)}\bigl(Y_{s}^{p}-L_{s}
\bigr)^-dK_{s}^{n} \Biggr]\xrightarrow[n,p\to+\infty]{}0.
\]
Thus
\[
\mathbb{E} \Biggl[\int_{t}^{T}e^{\beta
A(s)}a^2(s)\big|Y_{s}^{n}-Y_{s}^p\big|^2ds+
\int_{t}^{T}e^{\beta
A(s)}\big|Z_{s}^{n}-Z_{s}^p\big|^2ds
\Biggr]\xrightarrow[n,p\to+\infty]{}0.
\]
Moreover, by the Burkholder--Davis--Gundy's inequality, one can derive that
\[
\mathbb{E} \Bigl[\sup\limits
_{0\leq t\leq T}e^{\beta
A(t)}\big|Y_{t}^{n}-Y_{t}^p\big|^{2}
\Bigr]\xrightarrow[n,p\to+\infty]{}0.
\]
Further, from the equation (\ref{rr}), we have also
\[
\mathbb{E} \Bigl[\sup\limits
_{0\leq t\leq T}\big|K_{t}^{n}-K_{t}^p\big|^{2}
\Bigr]\xrightarrow[n,p\to+\infty]{}0.
\]
Consequently there exists a pair of progressively measurable processes
$(Z,K)$ such that
\[
\mathbb{E}\int_0^Te^{\beta A(t)}\big|Z_{t}^{n}-Z_{t}\big|^{2}dt
+\mathbb{E}\sup_{0\leq t\leq T}\big|K_{t}^{n}-K_{t}\big|^{2}
\xrightarrow[n\to +\infty]{}0.
\]
Obviously the triplet $(Y,Z,K)$ satisfies (\ref{r}). It remains to
check the Skorokhod condition. We have just seen that the sequence
$(Y^n,K^n)$ tends to $(Y,K)$ uniformly in $t$ in probability. Then the
measure $dK^n$ tends to $dK$ weakly in probability, hence
\[
\int_0^T\bigl(Y^n_t-L_t
\bigr)dK^n_t\xrightarrow[n\to+\infty]{\mathbb{P}}\int
_0^T(Y_t-L_t)dK_t.
\]
We deduce from the equation (\ref{e11}) that $\int_0^T(Y^n_t-L_t)dK^n_t\leq0$,
$n\in\mathbb{N}$, which implies that $\int_0^T(Y_t-L_t)dK_t\leq0$.
On the other hand, since $Y_t\geq L_t$ then
$\int_0^T(Y_t-L_t)dK_t\geq0$.
Hence $\int_0^T(Y_t-L_t)dK_t=0$.
\begin{remark}[Special cases]
The coefficients $g^n(s,y)=g(s)-n(y-U_s)^+$ and $\tilde
{g}^n(s,y)=g(s)-n(y-U_s)$ are Lipschitz and satisfy
\begin{align*}
&\mathbb{E}\int_0^Te^{\beta A(s)} \biggl
\llvert \frac{g^n(s,0)}{a(s)} \biggr\rrvert ^{2}ds +\mathbb{E}\int
_0^Te^{\beta A(s)} \biggl\llvert
\frac{\tilde
{g}^n(s,0)}{a(s)} \biggr\rrvert ^{2}ds
\\
&\quad\leq 4\mathbb{E}\int_0^Te^{\beta A(s)}
\biggl\llvert \frac{g(s)}{a(s)} \biggr\rrvert ^{2}ds +
\frac{4n^2T}{\epsilon}\mathbb{E} \Bigl[\sup_{n\geq0}e^{2\beta
A(t)}\big|U_t^-\big|^2
\Bigr]<+\infty.
\end{align*}
Then the Reflected BSDEs (\ref{r1}) and (\ref{d2}) have a unique solution.
\end{remark}
\begin{thm}[Comparison theorem]\label{tc}
Let $(Y^1,Z^1,K^1)$ and $(Y^2,Z^2,K^2)$ be solutions to the Reflected
BSDE (\ref{r}) with data $(\xi^1,f^1,L)$ and $(\xi^2,f^2,L)$
respectively. If we have
\begin{itemize}
\item$f^1(t,y)\leq f^2(t,y)$ a.s. $\forall(t,y)$,
\item$\xi^1\leq\xi^2$ a.s.,
\end{itemize}
then $Y_t^1\leq Y_t^2$ and $K_t^1\geq K_t^2$ 
$\forall t\in[0,T]$ a.s.
\end{thm}
\begin{proof}
We consider the penalized equations relative to the Reflected BSDE with
data $(\xi^i,f^i,L)$ for $i=1,2$ and $n\in\mathbb{N}$, as follows
\begin{equation*}
Y_{t}^{n,i}=\xi^i +\int_{t}^{T}
f^i\bigl(s,Y_{s}^{n,i}\bigr)ds+n\int
_{t}^{T}\bigl(Y_{s}^{n,i}-L_{s}
\bigr)^{-}-\int_{t}^{T}Z_{s}^{n,i}dB_{s}.
\end{equation*}
Let $f_n^i(t,y):=f^i(t,y)+n(y-L_{s})^{-}$. So, by the comparison
theorem, we have $Y_{t}^{n,1}\leq Y_{t}^{n,2}$ for $t\leq T$.
Since $K_{t}^{n,i}=n\int_{0}^{t}(Y_{s}^{n,i}-L_{s})^{-}ds$ for $i=1,2$,
we deduce that $K_{t}^{n,1}\geq K_{t}^{n,2}$ for $t\leq T$.
But $Y_{t}^{n,i}\nearrow Y_{t}^{i}$ and $K_{t}^{n,i}\longrightarrow
K_{t}^{i}$ as $n\longrightarrow+\infty$ for $i=1,2$, and it follows that
$Y_t^1\leq Y_t^2$ and $K_t^1\geq K_t^2$ for $t\leq T$.
\end{proof}
\end{appendix}
%

\end{document}